\documentclass[10pt, twoside]{amsart}

\usepackage[utf8]{inputenc}
\usepackage{hyperref}

\usepackage{xpatch}
\makeatletter   
\xpatchcmd{\@tocline}
{\hfil\hbox to\@pnumwidth{\@tocpagenum{#7}}\par}
{\ifnum#1<0\hfill\else\dotfill\fi\hbox to\@pnumwidth{\@tocpagenum{#7}}\par}
{}{}
\makeatother    

\makeatletter
 \def\l@subsection{\@tocline{2}{0pt}{4pc}{6pc}{}}
\def\l@subsubsection{\@tocline{3}{0pt}{8pc}{8pc}{}}
 \makeatother

\usepackage{graphicx}              
\usepackage{amsmath, bm}               
\usepackage{amsfonts}              
\usepackage{amsthm}                
\usepackage{amssymb}
\usepackage{mathtools}             
\usepackage{stackengine}
\usepackage{soul}

\makeatletter
\newcommand{\opnorm}{\@ifstar\@opnorms\@opnorm}
\newcommand{\@opnorms}[1]{%
  \left|\mkern-1.5mu\left|\mkern-1.5mu\left|
   #1
  \right|\mkern-1.5mu\right|\mkern-1.5mu\right|
}
\newcommand{\@opnorm}[2][]{%
  \mathopen{#1|\mkern-1.5mu#1|\mkern-1.5mu#1|}
  #2
  \mathclose{#1|\mkern-1.5mu#1|\mkern-1.5mu#1|}
}
\makeatother

\usepackage[usenames, dvipsnames]{color}
\usepackage{tikz}
\usetikzlibrary{matrix}
\usepackage{multirow}
\usetikzlibrary{shapes}

\usepackage{enumitem} 

\usepackage[all]{xy}
\usepackage{amscd}

\allowdisplaybreaks
\usepackage{color}
\usepackage{fp}

\newcommand{\Z}{\mathbb{Z}}

\newcommand{\C}{\mathbb{C}}
\newcommand{\R}{\mathbb{R}}

\newcommand{\Li}{\mathcal{L}}

\newcommand{\cj}{\overline}
\newcommand{\op}{\mathrm}

\DeclarePairedDelimiterX{\normb}[1]{\lVert}{\rVert}{#1}

\numberwithin{equation}{section}
\newtheorem{theorem}{Theorem}[section]
\newtheorem{lemma}[theorem]{Lemma}
\newtheorem{proposition}[theorem]{Proposition}
\newtheorem{corollary}[theorem]{Corollary}

\let\origproofname\proofname
\renewcommand{\proofname}{\upshape\textbf{\origproofname}}

\newtheorem*{theorem*}{Theorem}

\makeatletter
\newtheorem*{rep@theorem}{\rep@title}
\newcommand{\newreptheorem}[2]{%
\newenvironment{rep#1}[1]{%
 \def\rep@title{#2 \ref{##1}}%
 \begin{rep@theorem}}%
 {\end{rep@theorem}}}
\makeatother

\newreptheorem{theorem}{Theorem}

\theoremstyle{definition}
\newtheorem{example}[theorem]{Example}
\newtheorem{definition}[theorem]{Definition}
\newtheorem{remark}[theorem]{Remark}

\newtheorem{ques}[theorem]{Question}

\theoremstyle{remark}
\newtheorem{claim}{Claim}
\usepackage{etoolbox}
\AtEndEnvironment{proof}{\setcounter{claim}{0}}

\newcommand{\bit}{\begin{itemize}}
\newcommand{\ei}{\end{itemize}}

\newcommand{\cn}{\mathbb{C}}

\newcommand{\zn}{\mathbb{Z}}

\newcommand{\te}{\text}
\renewcommand{\l}{\left}
\renewcommand{\r}{\right}

\newcommand{\ve}{\varepsilon}

\DeclareMathAlphabet{\mathmybb}{U}{bbold}{m}{n}
\newcommand{\one}{\mathmybb{1}}

\newcommand{\mc}[1]{\mathcal{#1}}

\newcounter{compitem}

\usepackage{indentfirst}

\begin{document}

\title{Multiplier algebras of $L^p$-operator algebras}

\author{Andrey Blinov}
\address{AB: }
\email[]{andrew.blinov@gmail.com}

\author{Alonso Delfín}
\address{AD: Department of Mathematics, University of Colorado, Boulder CO 80309, USA}
\email[]{alonso.delfin@colorado.edu}

\author{Ellen Weld}
\address{EW: Department of Mathematics, Sam Houston State University, Huntsville TX 77341, USA}
\email[]{elw028@shsu.edu}

\date{\today}


\begin{abstract}
It is known that the multiplier algebra of an approximately unital and nondegenerate $L^p$-operator 
algebra is again an $L^p$-operator algebra. 
In this paper we investigate examples that drop both hypotheses. 
In particular, we show that the multiplier algebra of $T_2^p$, the algebra of strictly upper triangular $2 \times 2$ matrices 
acting on $\ell_2^p$, is still an $L^p$-operator algebra for any $p$. 
To contrast this result, we first provide a thorough study of the augmentation ideal of $\ell^1(G)$ for a discrete group $G$. We use this ideal to define a family of nonapproximately unital degenerate $L^p$-operator algebras, $F_{0}^p(\Z/3\Z)$, whose multiplier algebras cannot be represented 
on any $L^q$-space for any $q \in [1, \infty)$ as long as $p \in [1, p_0] \cup [p_0', \infty)$, where $p_0=1.606$ and 
$p_0'$ is its Hölder conjugate. 
\end{abstract}

\maketitle

\tableofcontents

\section{Introduction}

In 1971, C. S. Herz studied certain operators acting on $L^p$-spaces (see \cite{CSH71}). In fact, Herz
constructed group algebras of operators acting on $L^p$-spaces, those were referred to 
as algebras of $p$-pseudofunctions (see \cite{CSH73} and also \cite{Gard21} for a modern approach). 
About 15 years ago, the study of these operators regained interest from the community thanks to the general research works of M. Daws, E. Gardella, N. C. Phillips, N. Spronk, and H. Thiel, among others (see for instance \cite{Daw10}, \cite{DawSp19},  \cite{GarThi19}, \cite{GarThi20}, \cite{GarThi22}, and \cite{ncp2012AC}). This area is now known as $L^p$-operator algebras. 
Given that $L^p$-operator algebras include the large and well studied class of C*-algebras, part of the research in this area focuses on understanding what C*-properties and constructions can and cannot be extended to the $L^p$-case. In particular it is worth mentioning that $L^2$-operator algebras include the class of C*-algebras, but also the larger class of nonselfadjoint operator algebras. Three big differences between C*-algebras and $L^p$-operator algebras, even when $p=2$, which will be exploited here are: (1) $L^p$-operator norms are not unique, (2) some $L^p$-operator algebras cannot be nondegenerately represented on any Banach space, and (3) some $L^p$-operator algebras do not 
have contractive approximate identities (which will be henceforth abbreviated as cai's). 
A usually hard problem in $L^p$-operator algebras, 
arising in some sense due to the difficulty of computing $L^p$-operator norms,  is to determine whether a 
given Banach algebra can be isometrically represented on an $L^p$-space. Different 
techniques have been employed to show that certain Banach algebras cannot be represented on $L^p$-spaces, see for instance Theorem 2.5 in \cite{GarThi16} and Example 4.7 in \cite{blph2020} where it is shown that the class of $L^p$-operator algebras is not closed under taking quotients when $p \neq 2$. 
As a main result of this paper, we look at certain algebras with no contractive approximate identities to conclude that, for certain values of $p \neq 2$, the class of $L^p$-operator algebras is not closed under taking multiplier algebras. 

This paper is organized as follows: We start with a brief review of the multiplier algebra construction (as double centralizers)
for general Banach algebras. In Section \ref{secMA}, we present main results for the construction of the multiplier algebra, $M(A)$, of a general Banach algebra $A$. 
It is known that if a Banach algebra $A$ has a cai and is nondegenerately representable on a Banach space $E$, then $M(A)$ can also be nondegenerately represented on $E$. In fact, in Corollary \ref{M(A)=dm}, we prove that for the known representation of $M(A)$ on $E$ the action is given by two-sided multipliers, agreeing with the well known C*-result.
Next, we investigate general properties 
of $M(A)$ when $A$ has a non-unital identity which means that $A$ an algebraic multiplicative identity with norm greater than 1 (see Definition \ref{ident&unit}).  This condition automatically prevents $A$ from having a cai and also from being nondegenerately represented on any Banach space, making these algebras suitable candidates to investigate whether their multiplier algebras can be nondegenerately represented on Banach spaces.

Starting in Section \ref{MainR} and onward, we work only with multiplier algebras of 
 $L^p$-operator algebras. As before, when $A$ is a nondegenerate 
 $L^p$-operator algebra with a cai, its multiplier algebra 
 is also a nondegenerate 
 $L^p$-operator algebra. 
We begin Section \ref{MainR} by exploring whether these assumptions are necessary for the multiplier algebra of an $L^p$-operator algebra to be itself an $L^p$-operator algebra. Consequently, we produce two contrasting examples 
 dealing with the multiplier algebra of
 degenerate $L^p$-operator algebras without cai's: 
 
 The first one, explored in Subsection \ref{UTM}, 
 deals with strictly upper triangular matrices, denoted $T^p_d$, acting on $\ell_d^p=\ell^p(\{1, \ldots, d\})$.
 When $d=2$, even though $T_2^p$ is itself a degenerate $L^p$-operator algebra without a cai, we have the following main result:
 \begin{reptheorem}{MT2p} 
 For any $p \in [1, \infty)$, $M(T_2^p)$ is isometrically  isomorphic to $\C^2$ with the supremum norm. In particular, for any $p, q \in [1, \infty)$, $M(T_2^p)$ is an $L^q$-operator algebra that is nondegenerately representable on $\ell_2^q$. 
 \end{reptheorem} The second one is explored in Subsection \ref{AugId}, where we focus on $\ell_0^1(G)$, the augmentation ideal of $\ell^1(G)$ for a discrete group $G$. As an algebra, $\ell_0^1(G)$ is
an $L^1$-operator algebra that has an
identity with norm strictly greater than $1$ when $G$ is finite and $\op{card}(G)>2$.
The augmentation ideal  $\ell_0^1(G)$ is defined as the kernel of the map $\ell^1(G) \to \C $ given by $a \mapsto \sum_{g \in G} a(g)$. 
We then study the main properties of this ideal and show that, when 
$G$ is finite, $\ell_0^1(G)$ has no cai and cannot be nondegenerately represented on any Banach space (see Proposition \ref{finunit}\eqref{nocainoND}).  
For our main result, we construct a family of $L^p$-operator algebras 
whose multiplier algebras cannot be represented on any $L^q$-space for 
any $q \in [1,\infty)$ as long as $p \in [1, p_0] \cup [p_0', \infty)$, where $p_0=1.606$ and 
$p_0'$ is its Hölder conjugate. This is done by looking at the image of the augmentation ideal $\ell_0^1(\Z/3\Z)$ in the group algebra $F^p(\Z/3\Z)$ under the left regular representation $\lambda_p \colon \ell^1(\Z/3\Z) \to \Li(\ell^p(\Z/3\Z))$. That is, for each $p \in [1, \infty)$, we put
\[
F^p_{0}(\Z/3\Z)=\lambda_p(\ell_0^1(\Z/3\Z)),
\]
and then denote its multiplier algebra as $M_0^p(\Z/3\Z) = M(F^p_{0}(\Z/3\Z))$. 
The Banach algebra $F_{0}^p(\Z/3\Z)$ has an algebraic identity, denoted $\one_0$, whose 
norm is strict greater than $1$. However, in Proposition \ref{MultAlgF_0} we show that 
$M_0^p(\Z/3\Z)$ is simply $F_{0}^p(\Z/3\Z)$ equipped with a different
norm in which $\one_0$ now has norm $1$. 
Our main result in this subsection is then the following:
\begin{reptheorem}{NotLpOpAlg}
Let $p_0=1.606$ and let $p_0'=\frac{p_0}{p_0-1}$. Take any $p \in [1, p_0] \cup [p_0', \infty)$. Then 
$M_0^p(\Z/3\Z)$, 
the multiplier algebra of $F^p_{0}(\Z/3\Z)$, is not an $L^q$-operator algebra for any $q\in[1,\infty)$.
\end{reptheorem}
 The main ingredient in the proof of Theorem \ref{NotLpOpAlg} follows a technique previously used 
 by Blecher and Phillips in Section 4 of \cite{blph2020} to show that the class of $L^p$-operator algebras is not closed under taking quotients. To be more precise, we exhibit a bicontractive idempotent $e \in M_0^p(\Z/3\Z)$ such that $\one_0-2e$ is not an invertible isometry when $p \in [1, p_0] \cup [p_0', \infty)$, contradicting a structural 
 result of Bernau and Lacey, Theorem \ref{bicontractiveinvertibleisometry} below, regarding bicontractive idempotent acting on $L^p$-spaces.  Unfortunately, the bicontractive idempotent argument might not work for groups other that 
 $\Z/3\Z$. For instance, when $G=\Z/4\Z$ or $G=\Z/2\Z \oplus \Z/2\Z$,
the bicontractive idempotents in $M_0^p(G)=M(F^p_{0}(G))$
behave as expected for an $L^p$-operator algebra. 
The essential difficulty in deciding whether $M_0^p(G)$ is an $L^p$-operator algebra for a general finite group $G$ is that its norm is generally hard to compute (see \cite{HendOlsh2010} for instance).

In Section \ref{idempGelfand}, we restrict our attention to non-trivial finite abelian groups and show that when $n = \op{card}(G)>1$, the algebra $\ell_0^1(G)$ is algebraically 
 isomorphic to $\C^{n-1}$ with pointwise multiplication via the Gelfand map. 
 This yields a nice description of the idempotents of $M_0^p(\Z/n\Z)$, see Proposition \ref{idempotentsprop}, that 
 can be compared with Rudin's more general work for idempotents in $L^1(G)$ when $G$ is a locally compact group (see W. Rudin's work \cite{Rud63} and \cite{rudin1990fourier}). 
 
In the last section we present a brief discussion on different norms that make 
 $\C^n$ an $L^p$-operator algebra, posing a final question regarding what properties a norm on $\C^n$ must satisfy
to be an $L^p$-operator algebra. Answering this general question might provide a 
 tool to decide when $M_0^p(G)$ is representable on an $L^q$-space. 
 
We end our introduction with a brief overview of the general notation that will be used throughout the paper. 
 When $E$ is a Banach space, we denote by $\Li(E)$ the space of all bounded linear maps $E \to E$. 
For $p\in[1,\infty)$, we sometimes write $L^p(\mu)$ for $L^p(\Omega,\mu)$, the $L^p$-space of a measure space $(\Omega,\mu)$. In particular, when $\nu$ is the counting measure on $\Omega$ we simply write $\ell^p(\Omega)$ for $L^p(\Omega, \nu)$ and, for $d \in \Z_{\geq 2}$,
 we write $\ell_d^p$ for $\ell^p(\{1, \ldots, d\})$.

\section{Preliminaries}

We first establish standard definitions for representations of 
Banach algebras on Banach spaces. The definition 
of $L^p$-operator algebras will be given using 
this terminology. We then present basic facts about Banach algebras 
with either contractive approximate units or
with non-unital identities.
 
\begin{definition}
Let $A$ be a Banach algebra and let $E$ be a Banach space. A \textit{representation} of $A$ on 
$E$ is a continuous algebra homomorphism $\pi \colon A \to \Li(E)$. 
\begin{enumerate}
\item We say that $\pi$ is \textit{contractive} if 
$\| \pi(a)\| \leq \|a\|$ for all $a \in A$.
\item We say that $\pi$ is \textit{isometric} if 
$\| \pi(a)\| =\|a\|$ for all $a \in A$.
\item We say that $\pi$ is \textit{nondegenerate} if  
\[
\pi(A)E=
 \op{span}(\{ \pi(a)\xi \colon a \in A \mbox{ and } \xi \in E \}),
\]
is dense in $E$, and that 
$A$ is \textit{nondegenerately representable} if 
it has a nondegenerate isometric representation. 
\end{enumerate}
\end{definition}

\begin{definition}\label{DefnLpOA}
Let $p\in[1,\infty)$. A Banach algebra $A$ is an \textit{$L^p$-operator algebra} if there is a measure space $(\Omega,\mu)$ and an isometric representation of $A$ on $L^p(\mu)$.
\end{definition}

\begin{definition}
Let $A$ be a Banach algebra. We say that $A$ has 
a \emph{contractive approximate identity (cai)} if 
there is a net $(e_\lambda)_{\lambda \in \Lambda}$ in $A$ 
such that $\| e_\lambda \| \leq 1$ for all $\lambda \in \Lambda$
and for all $a \in A$, 
\[
\lim_{\lambda \in \Lambda} \| ae_\lambda -a \|  =\lim_{\lambda \in \Lambda} \| e_\lambda a -a \| = 0.
\]
\end{definition}

\begin{lemma}\label{normsup}
Let $A$ be a Banach algebra with a cai.
Then for any $a \in A$,
\[
\| a \| = \sup_{\| b \| = 1} \| ab \| = \sup_{\| b\|=1} \|ba \|.
\]
\end{lemma}
\begin{proof}
Let $(e_\lambda)_{\lambda \in \Lambda}$
be a cai of $A$. Then for any $a \in A$,
\[
\| a \| \geq \sup_{\| b\|=1} \| ba\| 
\geq \sup_{\lambda \in \Lambda} \frac{\| e_\lambda a\|}{\|e_\lambda\|}
 \geq \sup_{\lambda \in \Lambda} \| e_\lambda a\| \geq  \|a\|.
\]
The other equality is proved analogously. 
\end{proof}

\begin{definition}\label{ident&unit}
Let $A$ be a Banach algebra. We say that $A$ has 
an \emph{identity element} 
if there is an element $1_A\in A$ such that, for all $a\in A$,
\[
1_A\cdot a=a=a\cdot 1_A.
\]
In addition, if $\|1_A\|=1$, we call $1_A$ a \emph{unit} of $A$, 
otherwise we say $1_A$ is a \emph{non-unital identity} of $A$.
\end{definition}

We will work below with a Banach algebra that has a non-unital identity, see Subsection \ref{AugId}. The 
following two lemmas show that such Banach algebras cannot have cai's
and cannot be nondegenerately represented 
on any Banach space. 

\begin{lemma}\label{noCAI}
Let $A$ be a Banach algebra with a non-unital identity $1_A \in A$
(that is $\| 1_A \| \neq 1$). Then, $A$ cannot have 
a cai. 
\end{lemma}
\begin{proof}
It is clear that $\| 1_A \| > 1$.
Now suppose $A$ has a cai, say $(e_\lambda)_{\lambda \in \Lambda}$. 
Then, for each $\lambda \in \Lambda$
\[
0 \leq | \|e_\lambda\| - \|1_A\| | \leq  \| e_\lambda - 1_A \|= \| e_\lambda 1_A - 1_A \|.
\]
Therefore, by the squeeze theorem 
\[
\lim_{\lambda } | \|e_\lambda\| - \|1_A\| |=0.
\]
This implies that $\lim_{\lambda } \|e_\lambda\| = \|1_A\|$, whence $\| 1_A\| \leq 1$ which is a contradiction. 
\end{proof}

\begin{lemma}\label{noNDrep}
Let $A$ be a Banach algebra with a non-unital identity $1_A \in A$
(that is $\| 1_A \| \neq 1$). Then $A$ cannot be nondegenerately represented on any Banach space.
\end{lemma}
\begin{proof}
Assume on the contrary that there is a Banach space $E$ and an isometric nondegenerate representation 
$\pi \colon A \to \Li(E)$. 
Then, for any $a \in A$ and any $\xi \in E$ we have 
\[
 \pi(1_A)\pi(a)\xi = \pi(1_Aa)\xi  =\pi(a)\xi.
\]
Nondegeneracy now gives that $\pi(1_A)=\op{id}_{E}$, but this implies 
\[
1=\|\op{id}_{E}\| = \| \pi(1_A)\| =\| 1_A\| >1,
\]
a contradiction. 
\end{proof}

In Section \ref{MainR}, we will take the multiplier algebra of certain $L^p$-operator algebras in order to construct a family of unital Banach 
algebras that are not representable on any $L^q$-space for any $q \in [1, \infty)$.
Our technique, which was also used by Blecher and Phillips in Section 4 of \cite{blph2020}, will be to show that some bicontractive 
idempotents on the algebra do not behave as expected for 
$L^q$-operator algebras. We record precise definitions and 
known facts below. 

\begin{definition}\label{bicontractivedefn}
Let $A$ be a unital Banach algebra with unit $1_A$. 
\begin{enumerate}
\item  An idempotent $e\in A$ is \textit{bicontractive} if $\|e\|\leq 1$ and $\|1_A-e\|\leq 1$. 
\item An element $s\in A$ is an \textit{invertible isometry} if $\|s\|=1$ and $\|s^{-1}\|=1$.
\end{enumerate}
\end{definition}

The following is a combination of Theorem 2.1 in \cite{BerLa77} and part (c) of Lemma 2.29 in \cite{blph2020}.

\begin{theorem}\label{bicontractiveinvertibleisometry}
Let $p\in[1,\infty)$ and let $(\Omega,\mu)$ be a measure space. Then an idempotent $e\in \Li(L^p(\mu))$ is bicontractive if and only if $1-2e$ is an invertible isometry, where $1=\op{id}_{L^p(\mu)}$. 
\end{theorem}

\section{Multiplier Algebras}\label{secMA}

Below we present general definitions and some results about multiplier algebras of Banach algebras.  

Parts of following definition come from Section 2.5 in \cite{Dales2000BanachAA}. 

\begin{definition}\label{multAlg}
Let $A$ be a Banach algebra. 
\begin{enumerate}
\item A \emph{left multiplier of $A$} is a map $L \in \Li(A)$
satisfying 
\[
L(ab)=L(a)b \text{ \ for all \ } a,b \in A.
\]
The set of 
left multipliers of $A$ is denoted by $LM(A)$. 
\item A \emph{right multiplier of $A$} is a map $R \in \Li(A)$
satisfying 
\[
R(ab)=aR(b) \text{ \ for all \ } a,b \in A.
\]
 The set of 
right multipliers of $A$ is denoted by $RM(A)$. 
\item A \emph{double centralizer of $A$} is a pair $(L,R)$
with $L \in LM(A)$, $R \in RM(A)$, and 
\[
aL(b)=R(a)b \text{ \ for all \ } a,b \in A.
\]
We define $M(A)$, the \emph{multiplier algebra of $A$},
to be the subset of $\Li(A) \times \Li(A)^{\op{op}}$ (equipped with the supremum norm) consisting of double centralizers.
\end{enumerate}

\end{definition}

It is clear that $M(A)$ is a unital Banach subalgebra of  $\Li(A) \times \Li(A)^{\op{op}}$.

If $A$ is a C*-algebra, $M(A)$ is equivalently defined as
 the set of two sided multipliers of $A$ on any Hilbert space as long as $A$ acts nondegenerately on it;
 see Definition 2.2.2 in \cite{wegge-olsen_2004} for instance. 
 This will be also the case for a Banach algebra that has a cai and 
 that can be nondegenerately represented on a Banach space. 
 We start with the definition of a map from $A$ to $M(A)$
 that will be a canonical inclusion when $A$ has a cai. 
 
 \begin{definition}\label{iotaA}
 Let $A$ be a Banach algebra. For each $a \in A$ we define maps 
 $L_a \colon A \to A$ and $R_a \colon A \to A$ by left and right multiplication respectively, that is
 for any $b \in A$
 \[
L_a(b)=ab \ \text{ and } \ R_a(b)=ba.
 \]
 We get a map $\iota_A \colon A \to M(A)$ given by $\iota_A(a) =(L_a, R_a)$.
 \end{definition}

\begin{lemma}\label{iotaM}
Let $A$ be a Banach algebra with a cai, 
and let $(L,R) \in M(A)$. Then $\| L\| = \| R \|$. 
Furthermore, the map $\iota_A \colon A \to M(A)$ from 
Definition \ref{iotaA} 
is an isometric algebra homomorphism and $\iota_A(A)$ is a closed  two-sided ideal in $M(A)$. 
\end{lemma}
\begin{proof}
Lemma \ref{normsup} gives
\[
\| L(a) \| = \sup_{\|b\|=1} \|bL(a)\| = \sup_{\|b\|=1} \| R(b)a\| \leq \|R\|\|a\|.
\]
Thus, $\| L \| \leq \|R\|$. 
Similarly, 
\[
\| R(a) \| = \sup_{\|b\|=1} \|R(a)b\| = \sup_{\|b\|=1} \| aL(b)\| \leq \|a\|\|L\|,
\]
whence $\|R\| \leq \|L\|$ and therefore $\|L\|=\|R\|$. 
Note that $\iota_A$ is clearly a linear map. To check it is multiplicative 
take $a, b \in A$ and compute 
\[
\iota_A(ab)=(L_{ab}, R_{ab})=(L_aL_b, R_bR_a)=(L_a,R_a)(L_b,R_b)=\iota_A(a)\iota_A(b).
\]
Once again, Lemma \ref{normsup} gives $\|L_a\|=\|R_a\|=\|a\|$,
showing that $\iota_A$ is isometric and that $\iota_A(A)$ is closed in $M(A)$. 
Finally, direct computations show that for any $a \in A$ and 
any $(L,R) \in M(A)$,
$\iota_A(a)(L,R)=\iota_A(R(a))$ and $(L,R)\iota_A(a) = \iota_A(L(a))$, 
whence $\iota_A(A)$ is indeed a two-sided ideal in $M(A)$, finishing the proof.
\end{proof}

A well known and useful fact is that whenever $A$ is a Banach algebra 
with a cai that is nondegenerately represented on a Banach space $E$, then $M(A)$ 
can be nondegenerately represented on $E$. 
For convenience, we state such result below and refer the reader to 
Theorem 4.1 and Remark 4.2 in \cite{GarThi20} for a 
complete proof.  

\begin{theorem}\label{ExtAMA}
Let $A$ be a Banach algebra with a cai and let  
$\pi \colon A \to \Li(E)$ be a contractive nondegenerate 
representation of $A$ on a Banach space $E$. Then 
$\pi$ induces a unique nondegenerate, contractive, and unital representation 
$\widehat{\pi} \colon M(A) \to \Li(E)$ such 
that, if $\iota_A$ is as in Definition \ref{iotaA}, then $\pi= \widehat{\pi} \circ \iota_A$. 
Furthermore, $\widehat{\pi}$ is isometric when $\pi$ is. 
\end{theorem}

\begin{remark}
Theorem \ref{ExtAMA} can also be proved by either appropriately modifying the proof of Proposition 2.6.12 in \cite{BleLeMerd04} or by the methods described in Section 6 of \cite{John64}. We are grateful to both David Blecher and Hannes Thiel for pointing these references to us. An alternative way is to use 
Corollary 4.1.7 in \cite{Del2023}. This alternative proof uses that,  
just as in the C*-case, $M(A)$ is isometrically isomorphic to the two-sided multipliers of $A$ on any Banach space $E$ where $A$ acts nondegenerately. 
This fact is actually equivalent to Theorem \ref{ExtAMA}, as we see below in 
Corollary \ref{M(A)=dm}. 
\end{remark}

\begin{corollary}\label{M(A)=dm}
Let $A$ be a Banach algebra with a cai and that is 
nondegenerately represented on a Banach space $E$ via 
$\pi \colon A \to \Li(E)$. 
Then the Banach algebra $M(A)$ is isometrically isomorphic to 
\[
\{ t \in \Li(E) \colon t\pi(A) \subseteq \pi(A), \pi(A)t \subseteq  \pi(A)\},
\] 
the algebra of two sided multipliers of $\pi(A)$, via the map 
$\widehat{\pi}$ from Theorem \ref{ExtAMA}. 
\end{corollary}
\begin{proof}
 Theorem \ref{ExtAMA} gives that $\widehat{\pi} \colon M(A) \to \Li(E)$ is an isometric and unital representation satisfying $\widehat{\pi} \circ \iota_A = \pi$.
Set 
\[
B=\{ t \in \Li(E) \colon t\pi(A)\subseteq \pi(A), \pi(A)t \subseteq \pi(A)\}. 
\]
It only remains to show that the image of $\widehat{\pi}$ is indeed $B$.
To do so, notice that for any $(L,R) \in M(A)$ and $a \in A$ we have $\widehat{\pi}(L,R)\pi(a)= \pi(L(a))$ and $\pi(a)\hat{\pi}(L,R)= \pi(R(a))$, whence $\widehat{\pi}(M(A)) \subseteq B$. To check the reverse inclusion, for any $t \in B$ we define $L_t, R_t \in \Li(A)$ as follows:
\[
L_t(a)=\pi^{-1}(t\pi(a)), \ R_t(a)=\pi^{-1}(\pi(a)t), 
\]
where $\pi^{-1} \colon \pi(A) \to A$ is understood as the inverse of the invertible isometry $\pi \colon A \to \pi(A)$. 
Since $t \in B$, we have $(L_t, R_t) \in M(A)$. Furthermore, for any 
$a \in A$, $\xi \in E$ we have 
\[
\widehat{\pi}(L_t, R_t)\pi(a)\xi = \pi(L_t(a))\xi = t\pi(a)\xi.
\]
Hence, nondegeneracy implies that $\widehat{\pi}(L_t, R_t)=t$ and therefore 
we conclude that $B \subseteq \widehat{\pi}(M(A))$, finishing the proof.
\end{proof}

We will be mostly interested in $M(A)$ when $A$ is a Banach algebra that has a non-unital identity as in 
Definition \ref{ident&unit}. This implies that $A$ does not have 
a cai, as shown in Lemma \ref{noCAI}, whence the previous two results cannot be used. 
However, the following proposition shows that in such cases $M(A)$  is 
actually the Banach algebra $A$ equipped with a new norm.  

\begin{proposition}\label{MAidentity}
Let $A$ be a Banach algebra with identity $1_A \in A$, as in Definition \ref{ident&unit}, 
and let $\iota_A \colon A \to M(A)$ be as in Definition \ref{iotaA}. 
Then the map $\varphi \colon M(A) \to A$ given by 
\[
\varphi(L,R)=L(1_A)
\]
is an isometric isomorphism between the Banach algebras $M(A)$ and $(A, \opnorm{-})$, where $\opnorm{a}=\|\iota_A(a)\|$. In particular, $\opnorm{a}=\|a\|$ for all $a \in A$ when $1_A$ is a unit.
\end{proposition}
\begin{proof}
It is routine to check that $\varphi \colon M(A) \to A$  is a linear map. 
Further, it is clear that $\varphi(\iota_A(a)) =a$, 
whence $\varphi$ is surjective. If $\varphi(L,R)=0$, 
then for any $a \in A$ we get $L(a)=L(1_Aa)=L(1_A)a=\varphi(L,R)a=0$ and 
$R(a)=R(a)1_A=aL(1_A)=a\varphi(L,R)=0$, from where it follows that $\varphi$
is injective. Next, to show that $\varphi$ is multiplicative we
take any $(L, R), (L', R') \in M(A)$, and compute
\begin{align*}
\varphi\big( (L,R)(L',R') \big) & = \varphi( LL', R'R ) \\
& = L(L'(1_A)) \\
&=L(1_AL'(1_A))\\
&=L(1_A)L'(1_A) = \varphi(L,R)\varphi(L',R').
\end{align*}
Next, for any $(L,R) \in M(A)$ we obtain
 \[
 \opnorm{\varphi(L,R)}= \| \iota_A(L(1_A))\| = \| (L_{L(1_A)}, R_{L(1_A)})\| =\|(L,R)\|, 
 \]
so that $\varphi \colon M(A) \to A$ is indeed isometric when $A$ is equipped with $\opnorm{a}=\|\iota_A(a)\|$. Finally,
if $1_A$ is a unit, then $A$ has a trivial cai and therefore Lemma \ref{iotaM} implies that $\iota_A$ is an isometry.
\end{proof}

\begin{remark}\label{comM(A)}
Section 2 of \cite{wang1961} defines the multiplier algebra of a commutative Banach 
algebra $A$ without absolute zero divisors other than $0$ by 
\[
N(A)=\{ L \in \Li(A) \colon aL(b)=L(a)b \ \text{ for all } \ a,b \in A\}.
\] 
When $A$ is commutative and $(L,R) \in M(A)$, we have $L=R$ and 
$(L,R) \mapsto L$ is an isometric algebra isomorphism from $M(A)$ to $N(A)$. Hence, both definitions are equivalent. If $A$ is a commutative 
Banach algebra with identity $1_A \in A$ as in Definition \ref{ident&unit}, then $A$ has no absolute zero divisors other than $0$ and therefore Proposition \ref{MAidentity} above implies that $M(A)$ is isometrically isomorphic, as a Banach algebra, to $A$ equipped with the norm $a \mapsto \| L_a\|$. 
\end{remark}

\section{Main results}\label{MainR}

 In this section, we explore whether the conditions in Theorem \ref{ExtAMA} are necessary to guarantee that the multiplier algebra of an $L^p$-operator algebra is again an $L^p$-operator algebra. We do so via two contrasting examples presented in Subsections \ref{UTM} and \ref{AugId}.  On the one hand, we show that the multiplier algebra of $2 \times 2$ strictly upper triangular matrices acting on $\ell_2^p$ is an $L^q$-operator algebra for any $p, q \in [1,\infty)$. On the other hand, for each $p \in [1, \infty)$ we define a two dimensional $L^p$-operator algebra whose multiplier algebra is not an $L^q$-operator algebra for any $q\in[1,\infty)$ as long as $p \in [1,p_0] \cup [p_0',\infty)$ where $p_0=1.606$ and $p_0'$ is its Hölder conjugate given by $p_0'=\frac{p_0}{p_0-1}=2.65016501\ldots=2.\cj{6501}$. 

\subsection{Approximately Unital and Nondegenerate Case}

The following result is simply a consequence of Theorem \ref{ExtAMA} being a particular case of Corollary \ref{M(A)=dm}.

\begin{theorem}\label{M(A)isLp}
Let $A$ be an $L^p$-operator algebra that has a cai and such that 
there is a nondegenerate representation $\pi$ of $A$ on $L^p(\mu)$. 
Then $M(A)$ is an $L^p$-operator algebra that is isometrically 
isomorphic to the algebra of two sided multipliers of $\pi(A)$, that is $M(A) \cong \{ t \in L^p(\mu) \colon t\pi(a) \in \pi(A), \pi(a)t \in \pi(A) \text{ for all } a \in A\}$.
\end{theorem}

Theorem \ref{M(A)isLp} raises a natural question: Are the properties of 
having a cai and being nondegenerately represented necessary for $M(A)$
to be an $L^p$-operator algebra? Below, we present two examples of families of $L^p$-operator algebras that have neither property. In Subsection \ref{UTM}, we investigate strictly upper triangular $2 \times 2$ matrices and show its multiplier algebra
can be nondegenerately represented on an $L^p$-space. In Subsection \ref{AugId}, we investigate the group $L^p$-operator algebras associated to the augmentation ideal $\ell_0^1(\zn\slash 3\zn)$ and show that for certain values of $p$,  its multiplier algebra fails to be an $L^q$-operator algebra for any $q \in [1, \infty)$.  

\subsection{Strictly upper triangular matrices}\label{UTM}
Let $p \in [1, \infty)$. We start by considering $T_2^p$, the algebra of 
strictly upper triangular $2 \times 2$ matrices acting on $\ell_2^p$, that is,
\[
T_2^p = 
\left\{
\begin{pmatrix}
0 & z \\
0 & 0
\end{pmatrix}
\colon z \in \C \right\} \subset M_2^p(\C) = \Li(\ell_2^p).
\]
This is a degenerate $L^p$-operator algebra without identity, 
unit, nor cai. 
In fact, $T_2^p$ is isometrically isomorphic to $\C$ with the trivial multiplication $zw=0$
(see Example 3.6 in \cite{Gard21}). Notice that $T^p_2$ is commutative, but all its elements are absolute 
zero divisors, whence Remark \ref{comM(A)} does not apply. However, it is clear that 
$\Li(T_2^p)=\Li(\C)\cong \C$ and that $aL(b)=0=R(a)b$ for any $(L,R) \in M(T_2^p)$ and any 
$a,b \in T_2^p$. This, together with Definition \ref{multAlg}, imply at once that 
\[
M(T_2^p)= LM(T_2^p) \times  RM(T_2^p).
\] 
Furthermore, notice that $LM(T_2^p)=RM(T_2^p)=\Li(T_2^p)\cong \C$.
Therefore, $M(T_2^p)$ is isometrically isomorphic to $\C^2$ with the supremum norm, which is in turn isometrically identified with $C(\{1,2\})$, the space of continuous functions on $\{1,2\}$. 
Finally, for any $q \in [1, \infty)$, we have the map $\varphi_q \colon C(\{1,2\}) \to  \Li(\ell_2^q)$ given by 
\[
(\varphi_q(\xi)x)(j)=\xi(j)x(j),
\]
where $\xi \in C(\{1,2\})$, $x \in \ell_2^q$, and $j \in \{1,2\}$. 
This allows us to represent continuous functions on $\{1,2\}$ as multiplication operators
on $\ell_2^q$. It is well known and easy to see that $\varphi_q$ is an isometric nondegenerate representation. Therefore, we have shown the following result:
\begin{theorem}\label{MT2p}
For any $p \in [1, \infty)$, $M(T_2^p)$ is isometrically  isomorphic to $\C^2$ with the supremum norm. In particular, for any $p, q \in [1, \infty)$, $M(T_2^p)$ is an $L^q$-operator algebra that is nondegenerately representable on $\ell_2^q$.
\end{theorem}

A more complicated example is $T_3^p$, the algebra of 
strictly upper triangular $3 \times 3$ matrices acting on $\ell_3^p$. That is 
\[
T_3^p = 
\left\{
\begin{pmatrix}
0 & a_{12} & a_{13} \\
0 & 0 & a_{23}\\
0 & 0 & 0
\end{pmatrix}
\colon a_{12}, a_{13}, a_{23} \in \C \right\} \subset M_3^p(\C) = \Li(\ell_3^p),
\]
Notice that $T_3^p$ is identified with $\C^3$ with multiplication given by 
\[
ab=(a_{12} , a_{13}, a_{23} )(b_{12} , b_{13}, b_{23}) = (0, a_{12} b_{23}, 0 ).
\]
If $p=1$, then $T_3^p$ is isometrically isomorphic to 
$\C^3$ with norm 
\[
\| a \| = \op{max} \{|a_{12}|,|a_{13}|+ |a_{23}|\}.
\]
However, when $p \neq 1,2,\infty$ it is known that computing the $p$-operator norm of a matrix is an NP-hard problem (see \cite{HendOlsh2010}). This complicates the study 
of $M(T_3^p)$ and at this moment we have not investigated whether $M(T_3^p)$
can be isometrically represented on an $L^q$-space. 

\subsection{Augmentation Ideals}\label{AugId}

We will now construct, for certain values of $p$, a class of $L^p$-operator algebras
that have no cai, cannot be nondegenerately represented (on any Banach space), 
and whose multiplier algebra is not an $L^q$-operator algebra
for any $q\in [1, \infty)$. We begin with the general construction of the augmentation ideal for an arbitrary discrete group $G$ and eventually specialize to the finite case. In particular we will work with $G=\zn\slash 3\zn$. The augmentation ideal of a group ring is a well known algebraic object (see \cite[Section 11.2]{Johnson97}, \cite[Example 7, p. 245]{DummitFoote} for instance) and is easily generalized to operator algebras arising from discrete groups. We include details here for completeness.

Let $G$ be a discrete group with identity element 
denoted by $1_G \in G$. For any $g \in G$, the indicator function $\delta_g \colon G \to \C$ 
is given by
\[
\delta_{g}(h) = 
\begin{cases}
1 & \text{ if } g=h\\
0 & \text{ if } g \neq h\\
\end{cases}.
\]
For each $g \in G$, we define
the function $\Delta_g \colon G \to \C$ by 
\begin{equation}\label{DeltaCap}
\Delta_g = \delta_g - \delta_{1_G}.
\end{equation}
For each $p \in [1, \infty)$, recall that the Banach space  $\ell^p(G)$
is the set of complex valued functions $a\colon G \to \C$ satisfying 
\[
\|a\|_p^p= \sum_{g \in G} |a(g)|^p < \infty.
\]
Further, we define the convolution of any 
pair $(a, b) \in \ell^1(G) \times \ell^p(G)$ 
as the complex valued function
$a*b \colon G \to \C$ 
defined at each $g \in G$ by
\begin{equation}\label{convo}
(a*b)(g) = \sum_{h \in G} a(h)b(h^{-1}g).
\end{equation}
It is well known that $\| a*b \|_p \leq \|a\|_1\|b\|_p$, 
whence $a*b \in \ell^p(G)$. Moreover, when $p=1$, convolution gives rise to a multiplication on $\ell^1(G)$ 
that makes it a unital Banach algebra with unit $\delta_{1_G}$.

\begin{remark}\label{convNOT}
Given the large presence of multiplication by convolution in this paper, 
when $a \in \ell^1(G)$ and $b \in \ell^p(G)$ for some $p \in [1, \infty)$
we will simply write $ab$ instead of $a*b$ from Equation \eqref{convo}. It will  be understood 
by context when convolution is the multiplication that is being used. 
\end{remark}

It is also well known that
$C_c(G)$, the subspace of compactly supported 
functions on $G$, lies densely in $\ell^1(G)$ with respect to the 
$\|-\|_1$ norm. 
Clearly, $\delta_g \in C_c(G)$ for each $g \in G$. 

\begin{lemma}\label{iota0hom}
The map $\iota_0 \colon \ell^1(G) \to \C$ given by 
\begin{equation}\label{iota0}
\iota_0(a) = \sum_{g \in G} a(g)
\end{equation}
is a contractive algebra homomorphism. 
\end{lemma}
\begin{proof}
It is clear that $\iota_0$ is a well-defined linear map and that $|\iota_0(a)| \leq \|a\|_1$, 
whence $\| \iota_0 \| \leq 1$. 
By Fubini's Theorem, we see that
\[
\iota_0(ab)=\sum_{g \in G} (ab)(g) = \sum_{h \in G} a(h)\sum_{g \in G} b(h^{-1}g) = \iota_0(a)\iota_0(b),
\]
finishing the proof. 
\end{proof}

The main object of interest for this subsection is the augmentation 
ideal of $\ell^1(G)$, which is defined as 
the kernel of the map $\iota_0 \colon \ell^1(G) \to \C$ from 
equation \eqref{iota0}. We record a precise definition below. 

\begin{definition}\label{AugmentationIdealDef}
Let $G$ be a discrete group. The \emph{augmentation ideal} of $\ell^1(G)$ 
is defined as 
\[
\ell^1_0(G) = \ker(\iota_0) = \left\{ a \in \ell^1(G) \colon \iota_0(a)=0\right\}.
\]
Similarly, we define $\C_0[G] \subseteq C_c(G)$ as 
\[
\C_0[G] = \left\{ a \in C_c(G) \colon \iota_0(a)=0\right\}.
\]
\end{definition}

Next we establish the main properties 
of $\ell^1_0(G)$, which we present as a series of Propositions. 

\begin{proposition}\label{PropDenseloAUG}
Let $G$ be a discrete group.
\begin{enumerate}
\item  The augmentation ideal $\ell_0^1(G)$
is a closed two-sided proper ideal of $\ell^1(G)$. \label{tsideal}
\item The subspace $\C_0[G]$ is dense in $\ell_0^1(G)$  with respect to the 
$\|-\|_1$ norm. \label{c0dense}
\item Let $\Delta_g$ be
as in \eqref{DeltaCap}. Then $\op{span}\{ \Delta_g \colon g \in G\}$ is 
dense in $\ell_0^1(G)$. In fact, $\op{span}\{ \Delta_g \colon g \in G\}=\C_0[G]$. \label{Deltadense}
\end{enumerate}
\end{proposition}
\begin{proof}
That $\ell_0^1(G)$ is a closed two-sided ideal follows
from Lemma \ref{iota0hom} above. Since $\delta_{1_G} \not\in \ell_0^1(G)$, 
we also get that $\ell_0^1(G)$ is proper, so part (\ref{tsideal}) is done. 
For part (\ref{c0dense}), let $a \in \ell_0^1(G)$ and let $\varepsilon>0$. 
Then, there is $b \in C_c(G)$ such that $\| a-b \|_1 < \frac{\varepsilon}{2}$ 
and observe that $b-\iota_0(b)\delta_{1_G} \in \C_0[G]$. Further, note that 
 $|\iota_0(b)|=|\iota_0(a-b)| \leq \| a-b \|_1$. 
Therefore, 
\[
\| a - (b-\iota_0(b)\delta_{1_G} ) \|_1 \leq \|a-b\|_1 + |\iota_0(b)| \leq 2  \|a-b\|_1 < \varepsilon.
\]
This proves that $\C_0[G]$ is dense in  $\ell_0^1(G)$, as wanted. Finally, for part (\ref{Deltadense}), note first that 
$\op{span}\{ \Delta_g \colon g \in G\} \subseteq \C_0[G]$ because for each $g \in G$, it is clear 
that $\Delta_g \in C_c(G)$ and $\iota_0(\Delta_g)=\iota_0(\delta_g)-\iota_0(\delta_{1_G})=1-1=0$. 
For the reverse inclusion, take any $a \in \C_0[G]$, that is $\iota_0(a)=0$ and $F=\op{supp}(a)$ is a finite subset 
of $G$. Therefore, 
\[
a=\sum_{g \in F} a(g)\delta_g =\left( \sum_{g \in F} a(g)\delta_g  \right) -\iota_0(a) \delta_{1_G} = \sum_{g \in F}a(g)\Delta_g.
\]
This proves that $a \in \op{span}\{\Delta_g \colon g \in G\}$ and we are done with part (\ref{Deltadense}). 
\end{proof}

The previous proposition shows that $\ell^1_0(G)$ is a Banach algebra in its own right with the $1$-norm inherited from $\ell^1(G)$. When $p \in [1, \infty)$, we will mainly be interested in the image of $\ell_0^p(G)$
in $F^p(G)$, the full $L^p$ operator algebra of $G$, which we define below in the discrete group case setting. 
We refer the interested reader to Definition 2.1 in \cite{GarThi15} 
for a general definition of $F^p(G)$ for locally compact groups.  

\begin{definition}\label{FullGPS_0}
Let $G$ be a discrete group and let $p \in [1, \infty)$. 
\begin{enumerate}
\item We  define the \emph{full group $L^p$-operator algebra of $G$}, denoted by 
$F^p(G)$, as the completion of $\ell^1(G)$ in the norm
\[
a \mapsto \| a \|_{F^p(G)} = \sup \{ \| \pi(a)\| \colon \pi \in \op{Rep}_{L^p}(G)\},
\]
where $\op{Rep}_{L^p}(G)$ 
is the class of all nondegenerate contractive representations of $\ell^1(G)$ on $L^p$-spaces. 
\item  We define \textit{the augmentation ideal of $F^p(G)$}, denoted by  \label{FullAug}
$F^p_{0}(G)$, as the completion of $\ell_0^1(G)$ in the $\|-\|_{F^p(G)} $ norm.
\end{enumerate} 
\end{definition}

The algebra $F^p(G)$ is in fact an $L^p$-operator algebra, this follows from the general crossed product case, see for instance 
Theorem 3.6 in \cite{ncp2013CP}. We next show that $F^p_{0}(G)$ 
is actually a closed proper ideal in $F^p(G)$, making it also 
an an $L^p$-operator algebra in its own right. 

\begin{proposition}\label{FopPROPER}
Let $G$ be a discrete group and let $p \in [1, \infty)$. Then $F^p_{0}(G)$
is a nonzero closed two-sided proper ideal in $F^p(G)$. 
\end{proposition}
\begin{proof}
That $F^p_{0}(G)$
is a nonzero closed two-sided ideal follows at once from the 
fact that $\ell_0^1(G)$ is an ideal in $\ell^1(G)$. 
To show that it is proper, it suffices to 
prove that $\delta_{1_G}$, the unit in $F^p(G)$, 
does not belong to $F_0^p(G)$. 
Assume on the contrary that there is a sequence $(a_n)_{n=1}^\infty$ in $\ell_0^1(G)$
such that $a_n$ converges to $\delta_{1_G}$ in $F^p(G)$. 
Next, by Lemma \ref{iota0hom}, $\iota_0$ is a contractive representation of $\ell^1(G)$ on the $L^p$-space $\ell_1^p=\C$, so it follows from definition of $\|-\|_{F^p(G)}$ that 
$|\iota_0(a)|\leq \|a\|_{F^p(G)}$ for all $a \in \ell_0^1(G)$. Therefore 
$\iota_0$ extends continuously to a contraction $\iota_0 \colon F^p(G) \to \C$.
In particular, since $\iota_0(a_n)=0$ for all $n \in \Z_{\geq 1}$, it follows that
\[
1=|\iota_0(\delta_{1_G})| =|\iota_0(\delta_{1_G})-\iota_0(a_n)|   \leq \| \delta_{1_G}-a_n\|_{F^p(G)}.
\]
Taking limit as $n \to \infty$ shows that $1 \leq 0$, a contradiction. Thus, $\delta_{1_G } \not \in F^p_{0}(G)$, as wanted. 
\end{proof}

For each $p \in [1, \infty)$, we will also let $\ell^1_0(G)$ act on $\ell^p(G)$
via left multiplication by convolution. That way the completion of the image
of  $\ell^1_0(G)$ in $\Li(\ell^p(G))$ becomes an $L^p$-operator algebra. In particular, 
when $p=1$, the algebra $\ell^1_0(G)$ is isometrically isomorphic to its image in $\Li(\ell^1(G))$, 
making $\ell^1_0(G)$ an $L^1$-operator algebra. 
More precisely, this will be done by looking at the image of $\ell^1_0(G)$
in $F^p_\op{r}(G)$, the reduced $L^p$-operator algebra of $G$. 
The algebra $F^p_\op{r}(G)$ was
first introduced by Herz in \cite[Section 8]{CSH73}, where it was 
called the algebra of $p$-pseudofunctions on $G$, and then reintroduced by 
Phillips in \cite[Section 3]{ncp2013CP} as a particular case
of reduced $L^p$-operator crossed products. For convenience we
state below the definition of $F^p_\op{r}(G)$ for $G$ discrete. 
We refer the interested reader to Definition 3.1 in \cite{GarThi15}
for the general definition of $F^p_\op{r}(G)$ when $G$ is locally compact. 

\begin{definition}\label{Fp-FP0}
Let $G$ be a discrete group and let $p \in [1, \infty)$.
\begin{enumerate}
\item  Let $\lambda_p \colon \ell^1(G) \to \Li(\ell^p(G))$ denote 
left multiplication by convolution; that is, for $a \in \ell^1(G)$, $b \in \ell^p(G)$, we put $\lambda_p(a)b=ab \in \ell^p(G)$ as in Equation \eqref{convo} (see Remark \ref{convNOT}). We define the \textit{reduced $L^p$-operator algebra of $G$}, denoted by $F^p_\op{r}(G)$, as the closure of $\lambda_p(\ell^1(G))$ in $\Li(\ell^p(G))$. 
\item We define \textit{the augmentation ideal of $F^p_\op{r}(G)$}, denoted by 
$F^p_{\op{r},0}(G)$, as the closure of $\lambda_p(\ell_0^1(G))$ in $\Li(\ell^p(G))$. 
\end{enumerate}
\end{definition}

Since $G$ is discrete, $\|a\|_p \leq \|a\|_1$ for all $a \in \ell^1(G)$, in fact the canonical inclusion $\ell^1(G) \hookrightarrow \ell^p(G)$ has norm 1.
Furthermore, observe that for any $a \in \ell^1(G)$ we have
\begin{equation}\label{Fpnorms_Ineq}
\| a \|_p \leq \| a \|_{F^p_r(G)} \leq \| a \|_{F^p(G)} \leq \|a\|_1.
\end{equation}
Hence, when $p=1$, the map $\lambda_1 \colon \ell^1(G) \to \Li(\ell^1(G))$ is isometric
and therefore we have isometric isomorphisms of Banach algebras 
\begin{equation}\label{Fp=1}
F^1(G) \cong F^1_\op{r}(G)\cong\ell^1(G) \text{ and  } F_0^1(G) \cong F^1_{\op{r},0}(G) \cong \ell_0^1(G).
\end{equation}
For general $p \in [1, \infty)$, the map $\lambda_p \colon 
\ell^1(G) \to \Li(\ell^p(G))$ is an injective contractive nondegenerate representation of 
$\ell^1(G)$ on $\ell^p(G)$. Thus, $\lambda_p(\ell^1(G))$ is a dense subalgebra of $F^p_{\op{r}}(G)$ that is algebraically isomorphic to $\ell^1(G)$. Moreover, by definition 
$F^p_\op{r}(G)$ is a closed unital subalgebra of $\Li(\ell^p(G))$, making it 
an $L^p$-operator algebra as in Definition \ref{DefnLpOA}. 

\begin{remark} 
As in Proposition \ref{FopPROPER}, $F^p_{\op{r},0}(G)$
is a nonzero closed two-sided ideal in $F_{\op{r}}^p(G)$ and therefore an $L^p$-operator algebra. However, 
unlike Proposition \ref{FopPROPER}, we no longer have 
that the ideal $F_{\op{r}, 0}^p(G)$ is proper in general. This is explained due to the fact that 
in many instances $F^p_r(G)$ is a simple Banach algebra. 
Indeed, by Theorem 2.5 in  \cite{phillips2019simplicityreducedgroupbanach}, if $F^2_\op{r}(G)$ is simple, then  $F^p_r(G)$ is simple for any $p \in (1,\infty)$. In particular, for $p \in (1, \infty)$,
$F_{\op{r}}^p(G)$ will be simple when $G$ has the Powers property (see also Theorem 3.7 in \cite{pooya2015simplereducedlpoperator}). 
Thus, we can guarantee that for any $n\in \Z_{>1}$, $F^p_{r,0}(\mathbb{F}_n) =  F^p_r(\mathbb{F}_n)$ where $\mathbb{F}_n$ is the free group in $n$ generators. 
\end{remark}

However, Theorem 3.7 in \cite{GarThi15} shows that, for any $p \in (1,\infty)$, 
if $G$ is an amenable group, then $F^p(G)$ is isometrically isomorphic to $F_{\op{r}}^p(G)$ as 
Banach algebras. 
Therefore, this fact combined with Equation \eqref{Fp=1} give at 
once the following result:

\begin{proposition}\label{AugL1OA}
Let $G$ be a discrete group and let $p \in [1, \infty)$. If either $p=1$ or $G$ is amenable, then $F_{\op{r},0}^p(G)$ is a proper ideal in $F_{\op{r}}^p(G)$ that is degenerately represented on $\ell^p(G)$.
\end{proposition}

In any case, it is more natural to work with the ideal $F_0^p(G)$, rather than 
$F_{\op{r},0}^p(G)$, 
for $F_0^p(G)$ is guaranteed to be a proper ideal for any discrete group $G$. 
Moreover, for a finite nontrivial group $G$, we will show in Corollary \ref{F_0DEG} below that $F^p_{0}(G)$ cannot be nondegenerately represented on any Banach space, already voiding one of the assumptions in 
Theorem \ref{M(A)isLp}. Part of our work below is to show that, for a finite nontrivial group $G$, the assumption of 
having a cai is also not met for $F^p_{0}(G)$. 

\begin{remark} We mention that since finite groups are amenable, 
$F^p_{0}(G)$ and $F^p_{\op{r},0}(G)$ are isometrically isomorphic as Banach algebras (see Theorem 3.7 in \cite{GarThi15}). Therefore, if $G$ is a nontrivial finite group, we will keep using the simpler notation $F^p_{0}(G)$, but for
convenience we will use that $F_0^p(G)$ can be defined, in this case, as the closure of $\lambda_p(\ell_0^1(G))$ 
in $\Li(\ell^p(G))$. 
\end{remark}

For the rest of the paper, we will concentrate on the properties of $F^p_{0}(G)$
when $G$ is a finite nontrivial group.

\begin{proposition}\label{finunit}
Let $G$ be a finite group 
with $\op{card}(G) \geq 2$ and let $p \in [1, \infty)$. Then, for 
$n=\op{card}(G)$, we have:
\begin{enumerate}
\item as vector spaces,  \label{finunit1}
\[
F^p_{0}(G) = \ell_0^1(G)=\op{span}\{\Delta_g \colon g \in G\}=\C_0[G] \cong \C^{n-1};
\]
\item for any $a \in \ell_0^1(G)$, \label{finunitKil}
\[
a\left(\sum_{g \in G} \delta_g \right)  = \left(\sum_{g \in G} \delta_g \right)a=0;
\]
\item $F^p_{0}(G)$ has an identity element (see Definition \ref{ident&unit}) \label{finunit2}
given by 
\[
\one_0=-\frac{1}{n}\sum_{g\in G}\Delta_g;
\]
\item $\dfrac{((n-1)^p+n-1)^{1/p}}{n} \leq \|\one_0\|_{F^p(G)} \leq 2-\dfrac{2}{n};$\label{finunit3.0}
\item $\|\one_0\|_1=2-\dfrac{2}{n}$; \label{finunit3}
\item $\one_0$ is not a unit in $F^1_{0}(G)$ when $n > 2$  (see Definition \ref{ident&unit}); \label{finunit4}
\item $F^1_{0}(G)$ has no cai and cannot be nondegenerately represented on any Banach space when $n> 2$. \label{nocainoND}
\end{enumerate}
\end{proposition}
\begin{proof}
In what follows we write $n=\op{card}(G)$ and $G=\{g_1, \ldots, g_n\}$ where $g_1=1_{G}$. 
Since $G$ is finite, $C_c(G)=\ell^1(G)$ and $\ell^1(G) \cong \C^{n}$ as vector spaces via the map $\delta_{g_j} \mapsto \mathbf{e}_j \in \C^n$ for $j=1, \ldots, n$. Thus using Proposition \ref{PropDenseloAUG} we now get $\C_0[G]=\ell^1_0(G)$ and $\ell^1_0(G) \cong \C^{n-1}$ as vector spaces via the map $\Delta_{g_j} \mapsto \mathbf{e}_{j-1} \in \C^{n-1}$ for $j=2, \ldots, n$. Then $\ell_0^1(G)$ is a finite dimensional vector space and therefore already complete in the $\|-\|_{F^p(G)}$ norm. 
This gives that, for any $p \in [1,\infty)$, $F^p_{0}(G)=\ell_0^1(G)$ which is in turn identified with $\C^{n-1}$, equipped with a different norm for each $p$, proving part (\ref{finunit1}). Next, to show parts \eqref{finunitKil} and (\ref{finunit2}), we take any $a \in \ell_0^1(G)$ so that  $\iota_0(a)=0$. Then,
since $(a\delta_g)(h)=a(hg^{-1})$ for any $g,h \in G$, it follows that 
for any $k \in \{1, \ldots, n\}$
\[
\left(a\left(\sum_{j=1}^n \delta_{g_j} \right)\right)(g_k)=\sum_{j=1}^n a(g_kg_j^{-1})= \iota_0(a)=0,
\]
proving that $a\left(\sum_{j=1}^n \delta_{g_j}\right)=0$. A similar computation shows that 
$\left(\sum_{j=1}^n \delta_{g_j}\right)a=0$, so part \eqref{finunitKil} follows. For part (\ref{finunit2}), notice that for any $j,k \in \{1, \ldots, n\}$,
\[
(a\Delta_{g_j})(g_k) = (a\delta_{g_j})(g_k)-a(g_k) = a(g_kg_j^{-1})-a(g_k).
\]
Therefore, for any $k \in \{1, \ldots, n\}$ we obtain
\[
(a\one_0)(g_k) 
=\frac{1}{n} \sum_{j=1}^n \left( -a(g_kg_j^{-1}) +a(g_k)\right) = \frac{1}{n}(-\iota_0(a)+na(g_k))=a(g_k),
\]
from where it follows that $a\one_0=a$. An analogous computation gives that $\one_0a=a$ 
and therefore part (\ref{finunit2}) is done. Part (\ref{finunit3.0}) follows at once from 
\eqref{Fpnorms_Ineq} and the following direct computation for each $p \in[1,\infty)$:
\[
\|\one_0\|_p
=\frac{1}{n}\left\| \sum_{j=2}^n(\delta_{g_j}-\delta_{g_1}) \right\|_p
=\frac{1}{n}\Big( |n-1|^p +(n-1)|-1|^p\Big)^{1/p}.
\]
Part (\ref{finunit3}) is simply  (\ref{finunit3.0}) when $p=1$.  
Next, notice that part (\ref{finunit4}) is implied by part (\ref{finunit3}) for 
$\| \one_0\|_1 > 1$ when $n >2$.
Finally, when $n>2$, part  (\ref{finunit4}) together with Lemma \ref{noCAI} shows that $F^1_{0}(G)$ does not have a cai.
Combined with Lemma \ref{noNDrep}, this shows that $F^1_{0}(G)$ cannot be nondegenerately represented 
on any Banach space, so part \eqref{nocainoND} follows. 
\end{proof}

Our final goal for this subsection is to investigate whether $M(F^p_{0}(G))$, the multiplier algebra of 
$F^p_{0}(G)$ (see Definition \ref{Fp-FP0}), is an $L^p$-operator algebra when $G$ is finite and abelian.
Proposition \ref{finunit} \eqref{nocainoND} shows that, at least when $p=1,$ the hypotheses of Theorem \ref{M(A)isLp} are not met
for $A=F^p_{0}(G)$. We will show that
for all $p \in [1, \infty) \setminus \{2\}$ and any finite group 
$G$ with $\op{card}(G)>2$,   
$F^p_{0}(G)$ still does not meet the hypotheses in Theorem \ref{M(A)isLp}. 
In contrast to the algebra $T_2^p$ (see Subsection \ref{UTM}), 
we will see that for certain values of $p$, including 1, 
$M(F^p_{0}(\Z/3\Z))$ is not an $L^q$-operator algebra 
for any $q \in [1, \infty)$. 

We start with a computational lemma that will be needed twice 
in the rest of the paper. 

\begin{lemma}\label{COMPUTATIONALlemma}
Let $p \in (1, \infty) \setminus \{2\}$ and let $G$ be a finite group 
with $n=\op{card}(G) > 2$. Define $a_0 \colon G \to \C$ by 
\[
a_0(g)=\begin{cases}
n-1 & \text{ if } g=1_G \\
-1 &  \text{ if } g \neq 1_G
\end{cases}.
\]
 For each $\varepsilon \in \R$, we define $b_\varepsilon \colon G \to \C$ by 
\[
b_\varepsilon=a_0+\varepsilon\sum_{g\in G}\delta_g.
\]
Then 
\begin{enumerate}
\item $a_0 \in \ell^1_0(G)$ and $b_\varepsilon \in \ell^p(G) \setminus \ell^1_0(G)$ for any $\varepsilon \neq 0$; \label{a_0L0}
\item for any $\varepsilon \in \R$, $\one_0b_\varepsilon = a_0$ (when $\varepsilon\neq 0$ we think of $b_\varepsilon$ as a perturbation of $a_0$ that
leaves $\ell^1_0(G)$);\label{1be=a0}
\item there exists $\varepsilon_p \neq 0$ such that $\|a_0\|_p>\|b_{\varepsilon_p}\|_p$; \label{maximUm}
\item in particular, when $n=3$, we get \label{maximUmWHENn3}
\[
\frac{\|a_0\|_p}{\|b_{\varepsilon_p}\|_p} =  \frac{(2^p+2)(2^{\frac{1}{p-1}}+1)^p}{3^p(2^\frac{p}{p-1}+2)}>1.
\]
\end{enumerate}
\end{lemma}
\begin{proof}
Part \eqref{a_0L0} is readily verified. Part \eqref{1be=a0} follows at once from parts \eqref{finunitKil} and \eqref{finunit2} in Proposition \ref{finunit}. For part \eqref{maximUm}, 
observe that asking for $\|a_0\|_p>\|b_\varepsilon\|_p$ is equivalent to 
\begin{equation}\label{normsa_o/b_e}
(n-1)^p+n-1 > |n-1+\varepsilon|^p + (n-1)|\varepsilon-1|^p.
\end{equation}
Next, we define $f_p \colon (-1,1) \to \R$ by $f_p(\varepsilon)=\|a_0\|_p-\|b_\varepsilon\|_p$. That is, 
\[
f_p(\varepsilon)=(n-1)^p+n-1 -(n-1+\varepsilon)^p - (n-1)(1-\varepsilon)^p.
\]
We need to show that exists $\varepsilon_p \neq 0$ such that $f_p(\varepsilon_p)>0$. Clearly $f_p(0)=0$. Furthermore, it is routine to verify that $f_p$ is a differentiable concave function, that $f_p$ has a maximum at 
\[
\varepsilon_p= \frac{(n-1)^\frac{1}{p-1}-(n-1)}{(n-1)^\frac{1}{p-1}+1} \neq 0,
\]
and that 
\[
f_p'(0)= \frac{1}{n-1}((n-1)^2-(n-1)^p)p.
\]
Thus, since $n>2$, if $p \in (1,2)$ we have $f_p'(0)>0$ and $\varepsilon_p>0$, which 
gives $f_p(\varepsilon_p)>0$. 
Similarly, $p \in (2,\infty)$ implies $f'_p(0)<0$ and $\varepsilon_p<0$, 
from where it follows that $f_p(\varepsilon_p)>0$. This now gives 
part \eqref{maximUm}. Finally, part \eqref{maximUmWHENn3}
follows from plugging in $n=3$ and $\varepsilon_p$ in Equation \eqref{normsa_o/b_e}.
\end{proof}

We now obtain a generalization of part \eqref{finunit4} in Proposition \ref{finunit}
for any $p \neq 2$. The proof for $p=1$ was immediate, but since the right hand side of the inequality in part \eqref{finunit3.0} of Proposition \ref{finunit} is not always strictly greater than 1, a different argument is 
needed that works for any $p \in (1, \infty) \setminus \{2\}$ and 
any finite group $G$ with $\op{card}(G)>2$.

\begin{proposition}\label{one_0NU}
Let $p \in [1, \infty) \setminus \{2\}$ and let $G$ be a finite group 
with $\op{card}(G) > 2$.
Then $\one_0$ is a non-unital identity of $F^p_{0}(G)$ (see Definition \ref{ident&unit}). 
\end{proposition}
\begin{proof}
As before, we write $n=\op{card}(G)$ and $G=\{g_1, \ldots, g_n\}$ where $g_1=1_{G}$.
Proposition \ref{finunit}\eqref{finunit4} gives the result for $p=1$. 
Take $p\in (1,\infty)\setminus \{2\}$ and
recall that 
\[
\| \one_0 \|_{F^p(G)} = \|\lambda_p(\one_0)\|=\sup\left\{ \frac{\|\one_0b\|_p}{\|b\|_p} \colon b \in \ell^p(G), b \neq 0\right\}. 
\]
Now by Lemma \ref{COMPUTATIONALlemma} there is $a_0 \in \ell_0^1(G)$ and $b_{\varepsilon_p} \in \ell^p(G) \setminus \ell^1_0(G)$
such that $\one_0 b_{\varepsilon_p}=a_0$ and 
$\|a_0\|_p>\|b_{\varepsilon_p}\|_p$ for any $p \in (1,\infty)\setminus \{2\}$. 
Therefore, 
\[
\| \one_0 \|_{F^p(G)} \geq \frac{\| \one_0 b_{\varepsilon_p}\|_p}{\|b_{\varepsilon_p}\|_p}=\frac{\|a_0\|_p}{\| b_{\varepsilon_p}\|_p}>1,
\]
finishing the proof. 
\end{proof} 

\begin{remark}\label{1=1_p=2}
Notice that the argument given in the proof of Proposition \ref{one_0NU}
does not work when $p=2$. Indeed when $p=2$, with notation 
as in the proof of Lemma \ref{COMPUTATIONALlemma}, we get $\varepsilon_2=0$ 
and $f_2(0)=0$. In fact, $f_2$ is a parabola with vertex at the origin. 
The reason for this is that, when $p=2$, the C*-equation gives that 
$\| \one_0 \|_{F^2(G)}=1$. 
Furthermore, it follows from Theorem 3.18 of \cite{GarThi15}
that 
\[
\| \one_0 \|_{F^{p_1}(G)} \leq \| \one_0 \|_{F^{p_2}(G)} 
\]
when either $1 \leq p_1 < p_2 \leq 2 $ or $2 \leq p_2 < p_1 <\infty$. 
Proposition \ref{one_0NU} seems to suggest that $\| \one_0 \|_{F^{p}(G)}$ strictly decreases to $1$ as $p$ goes from $1$ to $2$ 
and that $\| \one_0 \|_{F^{p}(G)}$ is strictly increasing as $p$ goes from $2$ to infinity. However, we have not investigated these potential strict inequalities for all $p$ any further. 
\end{remark}

An immediate consequence of Proposition \ref{one_0NU}, whose 
proof is identical to part \eqref{nocainoND} in Proposition \ref{finunit}, 
is the following. 

\begin{corollary}\label{F_0DEG}
Let $p \in [1, \infty) \setminus \{2\}$ and let $G$ be a finite group 
with $\op{card}(G) > 2$. Then $F^p_{0}(G)$ has no cai and cannot be nondegenerately represented on any Banach space. 
\end{corollary}

Notice that the work done in Section \ref{secMA} already 
gives us a clear description of what $M(F_{0}^p(G))$ is
when $G$ is finite and abelian: 
\begin{proposition}\label{MultAlgF_0}
Let $G$ be a finite abelian group and let $p \in [1, \infty)$. 
Then, $M(F_{0}^p(G))$ is isometrically isomorphic to $(\ell_0^1(G), \opnorm{-}_p)$
where 
\[
\opnorm{a}_p=\|L_a\|_{F_{0}^p(G) \to F_{0}^p(G)}=\sup\{ \|ab\|_{F^p(G)} \colon b\in\ell_0^1(G),  \|b\|_{F^p(G)} =1 \}.
\]
\end{proposition}
\begin{proof}
 Since $G$ is abelian, 
$\ell_0^1(G)$ is commutative and, since $G$ is finite, $\ell_0^1(G)$  has a multiplicative identity 
and therefore the only absolute zero 
divisor in $\ell_0^1(G)$ is $0$. Thus, Remark \ref{comM(A)} together with Proposition \ref{finunit}\eqref{finunit1} give that $(L,R) \mapsto L$ is an isometric isomorphism 
from  $M( F_{0}^p(G) )$ to $\{ L \in \Li(F_{0}^p(G) ) \colon aL(b)=L(a)b\}$.  
Proposition \ref{MAidentity} now shows that $M(F_{0}^p(G) )$ is indeed isometrically isomorphic 
to $(\ell_0^1(G), \opnorm{-}_p)$ via $L_a \mapsto a$. 
\end{proof}

The multiplier algebra of $F_{0}^p(G)$ will be our main 
object of study for the rest of the section. We introduce a compact notation
that will be used henceforward:
\[
M_0^p(G)=M(F_{0}^p(G)).
\]
By Proposition \ref{MultAlgF_0}, when $G$ is finite abelian, then $M_0^p(G)$ is simply $\ell_0^1(G)$
equipped with the norm $\opnorm{-}_p$. 
When $p=1$, the norm in $M_0^1(G)=M(F_{0}^1(G) )=M(\ell_0^1(G))$ has a nicer description:
\[
\opnorm{a}_1=\sup\{ \|ab\|_{1} \colon b\in\ell_0^1(G),  \|b\|_{1} =1\}.
\]
This norm is equivalent to 
the $1$-norm in $\ell_0^1(G)$ via the 
following straightforward inequalities:
\begin{equation}\label{boundsLa}
\frac{\|a\|_1}{\|\one_0\|_1}\leq \opnorm{a}_1\leq \|a\|_1
\end{equation}
where we clearly have equality if and only if $\op{card}(G)=2$. In fact, 
these inequalities might be strict when $\op{card}(G)\neq 2$:
\begin{example}\label{someL_acomputation}
Let $G=\Z/3\Z=\{[0],[1],[2]\}$. Here, $\|\one_0\|_1=\frac{4}{3}$ and it is easy to check that $\opnorm{\Delta_{[j]}}_1=2=\|\Delta_{[j]}\|_1$ for 
$j\in \{1,2\}$. 
However, $\opnorm{\Delta_{[1]}+\Delta_{[2]}}_1=3\neq 4=\|\Delta_{[1]}+\Delta_{[2]}\|_1.$
\end{example}
For general $p \in [1, \infty)$ we get, using \eqref{Fpnorms_Ineq}, the following bounds 
for the norm in $M_0^p(G)$:
\begin{equation}\label{boundsLa_p}
\frac{\|a\|_p}{\|\one_0\|_1} \leq \frac{\|a\|_{F^p(G)}}{\|\one_0\|_{F^p(G)}}\leq \opnorm{a}_p \leq \|a\|_{F^p(G)}\leq \|a\|_1
\end{equation}

We end this section by explicitly showing, via a bicontractive idempotent argument, that for certain values of $p \in [1,\infty)\setminus\{2\}$, $M_0^p(\Z/3\Z)$
is not an $L^q$-operator algebra for any $q \in [1, \infty)$. Let $G=\zn\slash 3\zn=\{[0], [1], [2]\}$ and recall that $\ell_0^1(\Z/3\Z)$ is the span of elements $\Delta_{[1]}$ and $\Delta_{[2]}$. 
Furthermore, the multiplication in $\ell_0^1(\Z/3\Z)$ obeys the following table

\begin{center}
\begin{tabular}{ c|| c | c |}
 { } &$\Delta_{[1]}$&$\Delta_{[2]}$\\
 \hline
 \hline
$\Delta_{[1]}$ & $ -2\Delta_{[1]} + \Delta_{[2]}$ & $-\Delta_{[1]} -\Delta_{[2]} $\\
\hline
$\Delta_{[2]}$ & $-\Delta_{[1]} -\Delta_{[2]} $ & $ -\Delta_{[1]} + 2\Delta_{[2]}$\\
\hline
\end{tabular}
\end{center}

\noindent Recall that the element $\one_0 = -\tfrac{1}{3}\Delta_{[1]} - \tfrac{1}{3} \Delta_{[2]}$ is the identity in $\ell_0^1(\Z/3\Z)$.
For any $a \in \ell_0^1(\Z/3\Z)$, there are $\alpha_1, \alpha_2 \in \C$ such that $a = \alpha_1 \Delta_{[1]} + \alpha_{2} \Delta_{[2]}$. Therefore, 
for any $p \in [1, \infty)$ we have
\begin{equation*}
\|a\|_p =\|\alpha_1 \Delta_{[1]}  + \alpha_2 \Delta_{[2]}  \|_p = (|-\alpha_1 - \alpha_2|^p + |\alpha_1|^p + |\alpha_2|^p)^{1/p}.
\end{equation*}
In particular, $\| \one_0 \|_1 =  \tfrac{4}{3}$.

Proposition \ref{MultAlgF_0} shows that $M_0^p(\Z/3\Z)$ is isometrically isomorphic to $\ell_0^1(\Z/3\Z)$
equipped with the norm $\opnorm{-}_p$. 
For $a = \alpha_1 \Delta_{[1]} + \alpha_{2} \Delta_{[2]} \in M_0^p(\Z/3\Z)$, 
the inequalities in \eqref{boundsLa_p} become
\begin{equation} \label{BoundfromAbovefor3Cyclic}
\opnorm{a}_p \leq \|a\|_1=|-\alpha_1 - \alpha_2| + |\alpha_1| + |\alpha_2|
\end{equation}
and
\begin{equation} \label{BoundfromBelowfor3Cyclic}
\opnorm{a}_p \geq  \frac{\|a\|_p}{\|\one_0\|_1} = \frac{3}{4}(|-\alpha_1 - \alpha_2|^p + |\alpha_1|^p + |\alpha_2|^p)^{1/p}.
\end{equation}

We note that the norm on $M_0^p(\Z/3\Z)$ is an operator norm induced from a proper ideal of an $L^1$-space and is generally difficult to compute. Fortunately, there do exist certain elements $a$ for which we can compute $\opnorm{a}_p$ easily (see Example \ref{someL_acomputation}).
In particular, we can compute it for the idempotents of $M_0^p(\Z/3\Z)$.

\begin{proposition}\label{NonRepresentabilityofMultAlgof3CyclicAugIdeal}
Let $p \in [1, \infty)$ and let $\gamma = \tfrac{1}{3} \exp(2\pi i /3) \in \C$. 
We define $e = \gamma  \Delta_{[1]} + \cj{\gamma } \Delta_{[2]}$.
Then the following statements hold true:
\begin{enumerate}
\item \label{gamma} $|\gamma | = |\cj{\gamma }|=|\gamma +\cj{\gamma }|=\frac{1}{3}$, $|\gamma -\cj{\gamma }|=\frac{\sqrt{3}}{3}$;
\item \label{NonRepresentabilityofMultAlgof3CyclicAugIdeal_eIsIdempotent} 
$e^2=e$ and $\|e\|_1=1$;
\item \label{NonRepresentabilityofMultAlgof3CyclicAugIdeal_ComplementaryIdempotent} 
$\one_0 - e = \cj{\gamma } \Delta_{[1]}  + \gamma  \Delta_{[2]}$ and $\|\one_0-e\|_1=1$;
\item \label{NonRepresentabilityofMultAlgof3CyclicAugIdeal_NormofLe} 
$\opnorm{e}_p=1$;
\item \label{NonRepresentabilityofMultAlgof3CyclicAugIdeal_NormofLoneminuse} 
$\opnorm{\one_0-e}_p=1$;
\item \label{NonRepresentabilityofMultAlgof3CyclicAugIdeal_NormofSymmetry} 
$\opnorm{\one_0-2e}_1=\dfrac{2\sqrt{3}}{3}$;
\item \label{NonRepresentabilityofMultAlgof3CyclicAugIdeal_Norm2ofSymmetry}
$\opnorm{\one_0-2e}_2=1$;
\item \label{NonRepresentabilityofMultAlgof3CyclicAugIdeal_Normpleq2ofSymmetry}
$\opnorm{\one_0-2e}_p\leq \| \one_0 - 2e \|_{F^{p}(\Z/3\Z)}\leq \left(\frac{2\sqrt{3}}{3}\right)^{\frac{2}{p}-1}$ if $p \in [1,2]$;
\item
\label{NonRepresentabilityofMultAlgof3CyclicAugIdeal_Normpgeq2ofSymmetry}  $\opnorm{\one_0-2e}_p\leq \| \one_0 - 2e \|_{F^{p}(\Z/3\Z)} \leq \left(\frac{2\sqrt{3}}{3}\right)^{1-\frac{2}{p}}$ if $p \in [2, \infty)$.
\end{enumerate}
\end{proposition}
\begin{proof}
Each of the parts (\ref{gamma}), (\ref{NonRepresentabilityofMultAlgof3CyclicAugIdeal_eIsIdempotent}), and
(\ref{NonRepresentabilityofMultAlgof3CyclicAugIdeal_ComplementaryIdempotent}) are checked by direct computations. For part (\ref{NonRepresentabilityofMultAlgof3CyclicAugIdeal_NormofLe}), 
observe that the bounds \eqref{BoundfromAbovefor3Cyclic} and 
\eqref{BoundfromBelowfor3Cyclic} give
\begin{equation*}
0<\opnorm{e}_p \leq \|e\|_1=1.
\end{equation*}
Since $e^2=e$, submultiplicativity gives
\begin{equation*}
\opnorm{e}_p^2 \geq \opnorm{e^2}_p = \opnorm{e}_p.
\end{equation*}
 Dividing both sides by $\opnorm{e}_p$, we obtain that $\opnorm{e}_p \geq 1$. The proof of part (\ref{NonRepresentabilityofMultAlgof3CyclicAugIdeal_NormofLoneminuse}) 
is analogous to the part (\ref{NonRepresentabilityofMultAlgof3CyclicAugIdeal_NormofLe}). For part (\ref{NonRepresentabilityofMultAlgof3CyclicAugIdeal_NormofSymmetry}), first notice that 
\begin{equation*}
\|\one_0-2e\|_1=\|(\cj{\gamma }-\gamma)\Delta_{[1]} +(\gamma-\cj{\gamma })\Delta_{[2]} \|_1=|0|+2|\cj{\gamma }-\gamma|=\dfrac{2\sqrt{3}}{3}.
\end{equation*}
Then, the upper bound from \eqref{BoundfromAbovefor3Cyclic} gives us that 
\begin{equation*}
\opnorm{\one_0-2e}_1\leq \|\one_0-2e\|_1=\dfrac{2\sqrt{3}}{3}.
\end{equation*}
On other hand, since $\|\tfrac{1}{2}\Delta_{[1]}\|_1=1$, we 
compute
\begin{align*}
\|(\one_0-2e)(\tfrac{1}{2}\Delta_{[1]} )\|_1 &= \|((\cj{\gamma }-\gamma)\Delta_{[1]} +(\gamma-\cj{\gamma } )\Delta_{[2]} )(\tfrac{1}{2}\Delta_{[1]} )\|_1 \\
&= \frac{1}{2} \cdot |\cj{\gamma }-\gamma| \cdot \|(\Delta_{[1]} -\Delta_{[2]} )\Delta_{[1]} \|_1 \\
&= \frac{1}{2} \cdot \frac{\sqrt{3}}{3} \cdot \|-\Delta_{[1]} +2\Delta_{[2]} \|_1 \\
& = \frac{2\sqrt{3}}{3}.
\end{align*}
Therefore, $\opnorm{\one_0-2e}_1=\dfrac{2\sqrt{3}}{3}$. For part 
(\ref{NonRepresentabilityofMultAlgof3CyclicAugIdeal_Norm2ofSymmetry}), first notice that 
\begin{align*}
\one_0-2e
&=(\cj{\gamma }-\gamma)\Delta_{[1]} +(\gamma-\cj{\gamma })\Delta_{[2]} \\
&=(\cj{\gamma }-\gamma)(-\Delta_{[1]}+\Delta_{[2]})\\
&=\frac{-i\sqrt{3}}{3}(-\Delta_{[1]}+\Delta_{[2]}).
\end{align*}
Then, by computing $(\one_0-2e)\delta_{g}$ for each $g \in \Z/3\Z$,
we find that the standard matrix of $\one_0-2e$ is 
\begin{equation}\label{matrix_1-2e}
\frac{-i\sqrt{3}}{3}
\begin{pmatrix}
0 & 1 & -1 \\
-1 & 0 & 1 \\
1 & -1 & 0
\end{pmatrix}.
\end{equation}
Since the 
$\ell_d^2 \to \ell_d^2$ operator norm 
is the maximum singular value of the matrix, \eqref{matrix_1-2e} shows that that $\|\one_0-2e\|_{F^2(\Z/3\Z)}=1$. Combining this with the inner
inequalities in \eqref{boundsLa_p} and the fact that $\|\one_0\|_{F^2(\Z/3\Z)}=1$ (see Remark \ref{1=1_p=2})
we conclude that $\opnorm{\one_0-2e}_2=1$, as wanted. 
Finally, for parts \eqref{NonRepresentabilityofMultAlgof3CyclicAugIdeal_Normpleq2ofSymmetry} and 
\eqref{NonRepresentabilityofMultAlgof3CyclicAugIdeal_Normpgeq2ofSymmetry}, 
the result is clear for $p=1$ and $p=2$. Next assume that $p \in (1,2)$ with $\frac{1}{p}=(1-\theta)+\frac{\theta}{2}$
for $\theta \in (0,1)$. Then, the Riesz-Thorin interpolation theorem shows that 
\[
\| \one_0 - 2e \|_{F^p(\Z/3\Z)}
 \leq\| \one_0 - 2e \|_{F^1(\Z/3\Z)}^{1-\theta}\| \one_0 - 2e \|_{F^2(\Z/3\Z)}^\theta =  \left(\frac{2\sqrt{3}}{3}\right)^{\frac{2}{p}-1}.
\]
Now, if $p'=\frac{p}{1-p} \in (2,\infty)$, Proposition 4.11 in \cite{GarThi19}
shows that the ${F^{p'}(\Z/3\Z)}$ and ${F^{p}(\Z/3\Z)}$ norms 
of $\one_0 - 2e$
agree and therefore 
\[
\| \one_0 - 2e \|_{F^{p'}(\Z/3\Z)} \leq  \left(\frac{2\sqrt{3}}{3}\right)^{\frac{2}{p}-1} =  \left(\frac{2\sqrt{3}}{3}\right)^{1-\frac{2}{p'}}.
\]
The desired results now follow from the right inner bound in \eqref{boundsLa_p}.
\end{proof}

The previous proposition gives all the necessary 
tools to show that for a certain $p_0 \in (1,2)$, when $p \in [1, p_0) \cup (p_0', \infty)$, then the element 
$\one_0-2e \in M_0^p(\Z/3\Z)$ is not an invertible isometry, 
which is essentially the proof of the following main result. 

\begin{theorem}\label{NotLpOpAlg}
Let $p_0=1.606$ and let $p_0'=\frac{p_0}{p_0-1}$. Take any $p \in [1, p_0] \cup [p_0', \infty)$. Then 
$M_0^p(\Z/3\Z)$, 
the multiplier algebra of $F^p_{0}(\Z/3\Z)$, is not an $L^q$-operator algebra for any $q\in[1,\infty)$.
\end{theorem}

\begin{proof}
Suppose on the contrary that $M_0^p(\zn\slash 3\zn)$ is an $L^q$-operator 
algebra for some $q \in [1, \infty)$. Then, since $M_0^p(\zn\slash 3\zn)$ is a multiplier algebra it 
is unital (see Definition \ref{multAlg}), whence Proposition 3.7 in \cite{Gard21}
guarantees the existence of an isometric unital representation $\pi \colon M_0^p(\zn\slash 3\zn)\rightarrow \Li(L^q(\mu))$
for some measure space $(\Omega, \mu)$. 

Take $e = \gamma  \Delta_{[1]} + \cj{\gamma } \Delta_{[2]}$ as in Proposition \ref{NonRepresentabilityofMultAlgof3CyclicAugIdeal}. Parts \eqref{NonRepresentabilityofMultAlgof3CyclicAugIdeal_eIsIdempotent},\eqref{NonRepresentabilityofMultAlgof3CyclicAugIdeal_NormofLe}, 
and \eqref{NonRepresentabilityofMultAlgof3CyclicAugIdeal_NormofLoneminuse} in Proposition \ref{NonRepresentabilityofMultAlgof3CyclicAugIdeal} show that $e$ is a bicontractive element in 
$M_0^p(\zn\slash 3\zn)$. We claim the following:

\begin{claim}\label{claim}
$\opnorm{\one_0-2e}_p>1$ when $p \in [1,p_0]$ or $p \in [p_0', \infty)$. 
\end{claim}

Once the claim above is proven, it will follow that $\one_0-2e$ is not an invertible isometry in $M_0^p(\Z/3\Z)$. 
Thus, since $\pi$ is a unital isometry, we would have found a bicontractive idempotent $\pi(e)$ in $\Li(L^q(\mu))$
such that $1-2\pi(e)$ is not an invertible isometry when $p \in [1, p_0]\cup[p_0', \infty)$, contradicting Theorem \ref{bicontractiveinvertibleisometry}. 

Therefore it only remains to establish Claim \ref{claim}, which already follows when $p=1$ by Part \eqref{NonRepresentabilityofMultAlgof3CyclicAugIdeal_NormofSymmetry} in Proposition \ref{NonRepresentabilityofMultAlgof3CyclicAugIdeal}. 
Otherwise, observe that when $p \in (1, \infty) \setminus \{2\}$ we have 
\begin{equation}\label{keyBound}
\opnorm{\one_0 - 2e}_p \geq \frac{\|(\one_0 - 2e)(\one_0 - 2e)\|_{F^p(\Z/3\Z)}}{\| \one_0 - 2e \|_{F^p(\Z/3\Z)} }  = \frac{\|\one_0\|_{F^p(\Z/3\Z)}}{\| \one_0 - 2e \|_{F^p(\Z/3\Z)} } 
\end{equation}
Next, let $a_0 \in \ell_0^1(\Z/3\Z)$ and $b_{\varepsilon_p} \in \ell^p(\Z/3\Z)$ be as in part \eqref{maximUmWHENn3} of Lemma \ref{COMPUTATIONALlemma}. That is, for any $p \in (1, \infty) \setminus \{2\}$ we have 
\begin{equation}\label{keyBound2}
\|\one_0\|_{F^p(\Z/3\Z)}\geq \frac{\|a_0\|_p}{\|b_{\varepsilon_p}\|_p} =  \frac{(2^p+2)(2^{\frac{1}{p-1}}+1)^p}{3^p(2^\frac{p}{p-1}+2)}>1.
\end{equation}
Now take any $p \in (1,2)$ and set $p'=\frac{p}{p-1} \in (2, \infty)$. Proposition 4.11 in \cite{GarThi19} gives that the norms of $\|\one_0\|_{F^p(\Z/3\Z)}$ and $\|\one_0\|_{F^{p'}(\Z/3\Z)}$ agree, so together with
Part \eqref{NonRepresentabilityofMultAlgof3CyclicAugIdeal_Normpleq2ofSymmetry} 
in Proposition \ref{NonRepresentabilityofMultAlgof3CyclicAugIdeal}, combining the bounds
from equation \eqref{keyBound} and \eqref{keyBound2}, we now get
\[
\opnorm{\one_0 - 2e}_p \geq \frac{\|\one_0\|_{F^{p'}(\Z/3\Z)}}{(\frac{2\sqrt{3}}{3})^{\frac{2}{p}-1}} \geq \frac{\frac{\|a_0\|_{p'}}{\|b_{\varepsilon_{p'}}\|_{p'}}}{ (\frac{2\sqrt{3}}{3})^{\frac{2}{p}-1}} 
= \frac{(2^\frac{p}{p-1}+2)(2^{p-1}+1)^\frac{p}{p-1}}{3^\frac{p}{p-1}(2^p+2)(\frac{2\sqrt{3}}{3})^{\frac{2}{p}-1}}.
\]
Next we define $h  \colon (1,2) \to \R$ using the right hand side in the inequality above, that is 
\[
h(p)=\frac{(2^\frac{p}{p-1}+2)(2^{p-1}+1)^\frac{p}{p-1}}{3^\frac{p}{p-1}(2^p+2)(\frac{2\sqrt{3}}{3})^{\frac{2}{p}-1}}
=\frac{(2^{\frac{p}{p-1}}+2)\left(2^{p-1}+1\right)^{\frac{1}{p-1}}}{3^{\frac{1}{2}-\frac{1}{p}+\frac{p}{p-1}}\cdot2^{\frac{2}{p}}}.
\]
It can be checked that $p\mapsto h(p)$ is strictly decreasing on $(1,p_0]$, we omit the lengthy calculations. A direct computation now shows that $h(p_0)>1.000098$, so it follows that $\opnorm{\one_0 - 2e}_p >1$ for any $p \in (1,p_0]$. 
Similarly, when $p \in (2, \infty)$ we get 
using Part \eqref{NonRepresentabilityofMultAlgof3CyclicAugIdeal_Normpgeq2ofSymmetry} 
in Proposition \ref{NonRepresentabilityofMultAlgof3CyclicAugIdeal} 
\[
\opnorm{\one_0 - 2e}_p \geq  \frac{(2^p+2)(2^\frac{1}{p-1}+1)^p}{3^p(2^\frac{p}{p-1}+2)(\frac{2\sqrt{3}}{3})^{1-\frac{2}{p}}} = h\left(\frac{p}{p-1}\right).
\]
It follows from the $p \in (1, p_0]$ case that  $\opnorm{\one_0 - 2e}_p >1$ for any $p \in [p_0',\infty)$, 
proving Claim \ref{claim}, and therefore finishing the proof. 
\end{proof}


\begin{remark}
It is likely that, for all $p\in [1,\infty) \setminus \{2\}$, the algebra $M_0^p(\Z/3\Z)$ is not an $L^q$-operator algebra for any $q \in [1, \infty)$. A proof of this fact could potentially be obtained by the same method employed in the proof of Theorem \ref{NotLpOpAlg} above provided that we have either better bounds or exact values for the norms of 
$\| \one_0 \|_{ F^p(\Z/3\Z)}$ and $\| \one_0 -2e\|_{ F^p(\Z/3\Z)}$, which are generally hard to compute (see \cite{HendOlsh2010}). For instance, we believe that the actual value for $\| \one_0 -2e\|_{ F^p(\Z/3\Z)}$ when $p \in [1,2]$ 
is $\frac{\sqrt{3}}{3}(1+2^{p-1})^\frac{1}{p}$ which, if true, improves the value of $p_0$ to $p_0=1.675$.  
\end{remark}

%

\begin{remark}
Unfortunately, the bicontractive idempotent argument does not work when $G \in \{ \zn\slash 4\zn, \ \Z/2\Z \oplus \Z/2\Z\}$. Indeed, in this case if $e \in M_0(G)$ is any bicontractive idempotent, then it can be shown that $\one_0-2e \in M_0(G)$ is in fact an invertible isometry as expected for an $L^p$-operator algebra. Due to the difficulty of computing the norm of an element in $M_0^p(G)$, a different argument is likely needed to argue a version of Theorem \ref{NotLpOpAlg} for higher order groups. At this time, for a general finite group $G$ with $\op{card}(G)>3$ and $p \neq 2$, we do not know whether $M_0^p(G)$ can be isometrically represented on an $L^q$-space. 
\end{remark}

\section{Idempotents and Gelfand transform for $\ell_0^1(G)$}\label{idempGelfand}

In what follows we assume that $G$ is a finite abelian group. 
In Proposition \ref{MultAlgF_0}, for each $p \in [1, \infty)$,
we obtained a unital Banach algebra $M_0^p(G)=M(F^p_{0}(G))$
whose underlying space is the augmentation ideal 
$\ell_0^1(G)$ normed with $\opnorm{-}_p$. 
In addition, $M_0^p(G)$  is both a unital and a commutative Banach algebra when $G$ is abelian. 
When $\op{card}(G)>2$ and $p \neq 2$, notice that  $M_0^p(G)$ equipped with 
the inherited involution from $\ell^1(G)$ cannot be a C*-algebra. In fact, in such cases $M_0^p(G)$ is not isomorphically 
isometric to any closed subalgebra of $C(\Omega)$ where $\Omega=\op{Max}(M_0^p(G))$
because $\opnorm{\Delta_g}_p^2 \neq \opnorm{\Delta_g\Delta_g^*}_p$ for all $g \in G\setminus\{1_G\}$, 
so the Gelfand map is not isometric. We will see, however, 
that the Gelfand map is an algebraic isomorphism. 

\begin{lemma}\label{DeltaSystemSolns1}
For $n \in \Z_{>1}$, let $G=\zn\slash n\zn=\{[0],[1],\ldots,[n-1]\}$ and $(x_{[1]},\ldots,x_{[n-1]})\in\cn^{n-1}$ with $x_{[0]}=0$. Consider the system of equations
\begin{equation}\label{DeltaSystem}
x_{[j]+[k]}=x_{[j]}x_{[k]}+x_{[j]}+x_{[k]},\,\,\,\,1\leq j,k\leq n-1.
\end{equation}
Then the only nontrivial solutions for \eqref{DeltaSystem} are of the form 
\begin{equation}\label{DeltaSystemForm}
x_{[k]}=(x_{[1]}+1)^k-1 \text{ \ for all $1\leq k\leq n-1$ where \ } (x_{[1]}+1)^n=1.
\end{equation}
In particular, this means that  \eqref{DeltaSystem} has exactly $n-1$ non-trivial solutions.
\end{lemma}

\begin{proof}
Given a solution for \eqref{DeltaSystem}, we show that $x_{[k]}=(x_{[1]}+1)^k-1$ for all $1\leq k\leq n-1$. We do so 
by induction on $k$. 
When $k=1$,
$$x_{[1]}=(x_{[1]}+1)^1-1=x_{[1]}.$$

When $k=2$,
\[
x_{[2]}=x_{[1]+[1]}=x_{[1]}^2+2x_{[1]}\\
=(x_{[1]}+1)^2-1
\]

Suppose $x_{[j]}=(x_{[1]}+1)^j-1$. Then
\begin{align*}
x_{[j]+[1]}&=x_{[j]}x_{[1]}+x_{[j]}+x_{[1]}\\
&=\l((x_{[1]}+1)^j-1\r)x_{[1]}+\l((x_{[1]}+1)^j-1\r)+x_{[1]}\\
&=x_{[1]}(x_{[1]}+1)^j-x_{[1]}+x_{[1]}+(x_{[1]}+1)^j-1\\
&=(x_{[1]}+1)(x_{[1]}+1)^j-1\\
&=(x_{[1]}+1)^{j+1}-1
\end{align*}

Now since $x_{[0]}=0$, it follows that 
\[
0=x_{[0]}=x_{[n-1]+[1]}=x_{[n-1]}x_{[1]}+x_{[n-1]}+x_{[1]}=(x_{[1]}+1)^n-1,
\]
so indeed any solution of the system \eqref{DeltaSystem} is of the form \eqref{DeltaSystemForm}.

Next, we show that $x_{[k]}=(x_{[1]}+1)^k-1$ for all $1\leq k\leq n-1$ and $(x_{[1]}+1)^n=1$ satisfies \eqref{DeltaSystem}. Let $1\leq j,k\leq n-1$. Then,
\begin{align*}
x_{[j]}x_{[k]} & =\l((x_{[1]}+1)^j-1\r)\l((x_{[1]}+1)^k-1\r)\\
& = (x_{[1]}+1)^{j+k}-(x_{[1]}+1)^k-(x_{[1]}+1)^j+1,
\end{align*}
and since $x_{[j]}=(x_{[1]}+1)^j-1$ for any $1\leq j\leq n-1$, we conclude 
\[
x_{[j]}x_{[k]}+x_{[j]}+x_{[k]} =(x_{[1]}+1)^{j+k}-1=x_{[j+k]}=x_{[j]+[k]}.
\]
Clearly, $(x_{[1]}+1)^n=1$ satisfies $(x_{[1]}+1)^n-1=0=x_{[0]}$, so we are done. 
\end{proof}

Maschke's and Wedderburn's theorems (see Chapter 16 of \cite{DummitFoote} for an elementary introduction) imply that the only nilpotent element of $C_c(G)$ for $G$ finite abelian is $0 \in C_c(G)$. In particular, this means that $C_c(G)$ and all of its (two-sided) ideals have trivial nilradical. This leads to the following important observation.

\begin{lemma}\label{TrivialJacobsonRadical}
Let $I$ be any (two-sided) ideal in $\ell^1(G)$ for $G$ finite abelian. Then the Gelfand Transform $\Gamma_I \colon I\rightarrow C(\op{Max}\,(I))$ is an algebraic isomorphism.
\end{lemma}

We now begin with a technical lemma.

\begin{lemma}\label{ProdStructure}
Let $G$ be any discrete group and define for $m\in \Z_{>1}$
$$\mc{B}^m=\{(\ve_1,\ve_2,\ldots,\ve_m)\,\,:\,\,\ve_j\in\{0,1\}\te{ for all }1\leq j\leq m\}\setminus\{(0,0,\ldots,0)\}.$$
Then for any collection of $m$ elements $g_1,g_2,\ldots,g_m\in G\setminus\{1_G\}$,
\begin{equation}\label{ProdInduct}
\Delta_{g_1g_2\cdots g_m}=\sum_{(\ve_1,\ve_2,\ldots,\ve_m)\in\mc{B}^m}\Delta_{g_1}^{\ve_1}\Delta_{g_2}^{\ve_2}\cdots \Delta_{g_m}^{\ve_m}
\end{equation}
where we define $\Delta_{g_j}^0=1$ for all $1\leq j\leq m$. In particular, this means
\begin{equation}\label{ProdStrEqn}
\Delta_{gg'}=\Delta_g\Delta_{g'}+\Delta_g+\Delta_{g'}\te{ for all }g,g'\in G\setminus\{1_G\}.
\end{equation}
\end{lemma}

\begin{proof}
Fix $g,g'\in G\setminus\{1_G\}$. We argue by induction. A direct computation gives 
\begin{align*}
\Delta_{gg'}=\Delta_g\Delta_{g'}+\Delta_{g'}+\Delta_{g}
\end{align*}
Next, assume equation \eqref{ProdInduct} holds for the product of $m-1$ non-identity elements of $G$. Choose $g_1,g_2,\ldots,g_{m-1},g_m\in G\setminus\{1_G\}$. Then, using the induction hypothesis at the last step
\begin{align*}
\Delta_{g_1g_2\cdots g_{m-1}g_m}&=\Delta_{(g_1g_2\cdots g_{m-1})g_m}\\
&=\Delta_{g_1g_2\cdots g_{m-1}}\Delta_{g_m}+\Delta_{g_1g_2\cdots g_{m-1}}+\Delta_{g_m}\\
&=\sum_{(\ve_j)\in\mc{B}^m}\Delta_{g_1}^{\ve_1}\Delta_{g_2}^{\ve_2}\cdots\Delta_{g_{m-1}}^{\ve_{m-1}}\Delta_{g_m}^{\ve_m}.
\end{align*} 
This finishes the proof. 
\end{proof} 

We are now ready to prove the cyclic group case. 

\begin{proposition}\label{GelfandIsomCyclic}
Let $G=\zn\slash n\zn$ for $n \in \Z_{>1}$. Then $\ell_0^1(G)$ is algebraically isomorphic to $\cn^{n-1}$ equipped with pointwise multiplication. 
\end{proposition}

\begin{proof}
We have already established that $\ell_0^1(\zn\slash n\zn)$ is an ideal in $\ell^1(\zn\slash n\zn)$ and so we immediately have an algebraic isomorphism between $\ell_0^1(\zn\slash n\zn)$ and continuous functions on its maximal ideal space. Thus, 
it suffices to establish that the maximal ideal space of $\ell_0^1(\zn\slash n\zn)$ consists of 
$n-1$ elements. 

Let $\zn\slash n\zn=\{[0],[1],\ldots,[n-1]\}$. Thanks to Equation \eqref{ProdStrEqn}, any multiplicative linear functional $\omega \colon \ell_0^1(\zn\slash n\zn)\rightarrow \cn$ must satisfy 
\begin{equation}
\omega(\Delta_{[j]+[k]})=\omega(\Delta_{[j]})\omega(\Delta_{[k]})+\omega(\Delta_{[j]})+\omega(\Delta_{[k]})
\end{equation}
for all $1\leq j,k\leq n-1$. Letting $\omega(\Delta_{[j]})=x_{[j]}$, we recover the system described in Lemma \ref{DeltaSystemSolns1}. We observe that each solution of the system \eqref{DeltaSystem} is completely determined by the choice of $x_{[1]}$, which must be chosen to satisfy $(x_{[1]}+1)^n=1$. For each $1\leq j\leq n-1$, let
\begin{equation}\label{xij}
\omega_j(\Delta_{[1]})=\te{exp}(2\pi i\cdot j/n)-1=x_{j,[1]}.
\end{equation}
As discussed, this choice determines the remaining $\omega_j(\Delta_{[k]})=x_{j,[k]}$ and so we have a one-to-one correspondence between the multiplicative linear functionals of the ideal $\ell^1_0(\zn\slash n\zn)$ and the solutions of System \eqref{DeltaSystem}. Of course, $\op{Max}(\ell_0^1(\zn\slash n\zn))$ is in bijection with $\Omega = \{\omega_j\,:\,1\leq j\leq n-1\}$. Lemma \ref{TrivialJacobsonRadical} gives that 
$$\Gamma \colon \ell_0^1(\zn\slash n\zn)\rightarrow C(\Omega)$$
is an algebraic isomorphism. Finally since the canonical map 
$C(\Omega)   \to \C^{n-1}$ given by
\[
f \mapsto (f(\omega_1),f(\omega_2),\ldots,f(\omega_{n-1}))
\]
is also  an algebraic isomorphism, we have indeed shown that $\ell_0^1(\zn\slash n\zn)$ and 
$\C^{n-1}$ are algebraically isomorphic, completing the proof.
\end{proof}

We now extend Proposition \ref{GelfandIsomCyclic} to any finite abelian group. Consider 
\[
G=(\zn\slash n_1\zn) \times (\zn\slash n_2\zn) \times \cdots\times (\zn\slash n_m\zn),
\]
where $n_1, \ldots, n_m \in \Z_{>1}$. Then any $a\in\ell_0^1(G)$ may be written as
\begin{equation*}
a=\sum_{(g_1,g_2,\ldots,g_m)\in G}a(g_1,g_2,\ldots,g_m)\Delta_{(g_1,g_2,\ldots,g_m)}.
\end{equation*}
We note that $(g_1,g_2,\ldots,g_m)(g_1',g_2',\ldots,g_m')=(g_1+g_1',g_2+g_2',\ldots,g_m+g_m')$ and we will abuse notation by writing $g$ for $(0,\ldots,g,\ldots0)$ and $g_1g_2\cdots g_m$ for $(g_1,g_2,\ldots,g_m)$.

For each cyclic group $\zn\slash n_j\zn$, let $\{\omega_j^{(k)}\}_{0\leq k\leq n_j-1}$ be the  multiplicative linear functionals associated with $\op{Max}\,(\ell_0^1(\zn\slash n_j\zn))$ where $\omega_j^{(0)}$ represents the zero function on $\ell_0^1(\zn\slash n_j\zn$). We show that all multiplicative linear functionals of $\ell^1_0(G)$ are of the form
\begin{equation}\label{MultLinFuncProd}
\omega(\Delta_{g_1g_2\cdots g_m})=\sum_{(\ve_j)\in\mc{B}^m}\omega_1^{(k_1)}(\Delta_{g_1}^{\ve_1})\omega_2^{(k_2)}(\Delta_{g_2}^{\ve_2})\cdots\omega_m^{(k_m)}(\Delta_{g_m}^{\ve_m})
\end{equation}
for all $0\leq k_i\leq n_i,1\leq i\leq m$ where we interpret $\omega_{n_j}^{(k_j)}(\Delta_{g_k}^0)$ as equaling 1 when in a product with other nonzero terms.

Suppose $\omega \colon \ell_0^1(G)\rightarrow\cn$ is a multiplicative linear functional. When $\omega$ is restricted to linear sums of $\Delta_{g_j}$ for $g_j\in \zn\slash n_j\zn$, which we will denote by $\omega_j$, then $\omega_j$ is a multiplicative linear functional of $\ell_0^1(\zn\slash n_j\zn)$. Thus, there must exist some $k_j\in\{1,2,\ldots,n_j-1\}$ such that $\omega_j=\omega_j^{(k_j)}$. Equation \eqref{MultLinFuncProd} then follows by previous comments and Lemma \ref{ProdStructure}. Note that $\omega$ is the zero map exactly when all the $\omega_j^{(k_j)}$ are zero maps. Thus, we see that we have $n_1n_2\cdots n_m=\op{card}(G)$ possible choices for $\omega$, including the zero map.

This leads to the following theorem and corollary.

\begin{theorem}\label{GelfandIsom}
Let $G$ be a finite abelian group with $\op{card}(G)>1$. Then the Gelfand transform from $\ell_0^1(G)$ to $C(\{1,2,\ldots,n-1\})$ is an algebraic isomorphism. 
\end{theorem}

\begin{corollary}\label{GelfandIsom2}
Let $G$ be a finite abelian group with $\op{card}(G)>1$. Then the algebras $\ell_0^1(G)$, $\ell^1(\zn\slash (n-1)\zn)$, $C(\{1,2,\ldots,n-1\})$, and $\cn^{n-1}$ are all algebraically isomorphic.
\end{corollary}

Recall that we used a bicontractive idempotent as the primary tool to show that for 
certain values of $p \in [1, \infty) \setminus\{2\}$, the algebra  $M_0^p(\Z/3\Z)$ cannot be isometrically represented on any $L^q$-space. To this end, we present a description of the idempotents 
of $\ell_0^1(\zn\slash n\zn)$ and, in consequence, those of $M_0^p(\Z/n\Z)$ since the algebraic structures are the same. We start with a short definition. 

\begin{definition}\label{atilde}
Let $a \in \ell_0^1(\Z/n\Z)$ for $n \in \Z_{>1}$. Then, since $a$ is uniquely determined by $[a]=(a([1]), \ldots, a([n-1])) \in \C^{n-1}$, 
we define a canonical map $\ell_0^1(\Z/n\Z) \to \C^{n-1}$ via $a \mapsto [a]\in \C^{n-1}$.
\end{definition}

Let $n \in \Z_{>1}$. We define the \textit{translation matrix of $\ell_0^1(\zn\slash n\zn)$} by 
\begin{equation}\label{transmatrix}
X=\l(\begin{array}{cccc}
x_{1,[1]}& x_{1,[2]}&\cdots & x_{1,[n-1]}\\
x_{2,[1]}& x_{2,[2]}&\cdots & x_{2,[n-1]}\\
\vdots & \vdots & \ddots & \vdots\\
x_{n-1,[1]}& x_{n-1,[2]} & \cdots &x_{n-1,[n-1]}
\end{array}\r).
\end{equation}
Proposition \ref{GelfandIsomCyclic} implies that for any $v\in\cn^{n-1}$ there exists a unique $a\in \ell_0^1(\zn\slash n\zn)$ such that $X[a]=v$, where $[a]$ is as in Definition \ref{atilde}. In particular, if we let $\{e_k\}_{k=1}^{n-1}$ represent the standard basis of $\cn^{n-1}$, then all non-zero idempotents of $\cn^{n-1}$ are of the form
$$\ve_1e_1+\ve_2e_2+\cdots+\ve_{n-1}e_{n-1}\quad \te{for} \quad (\ve_1,\ve_2,\ldots,\ve_{n-1})\in\mc{B}^{n-1}$$
where $\mc{B}^{n-1}$ is as in Lemma \ref{ProdStructure}. Then our discussion above yields the following
Proposition.
\begin{proposition}\label{idempotentsprop}
Fix $n \in \Z_{>1}$ and let $a_k\in\ell_0^1(\zn\slash n\zn)$ be the unique element such that $X[a_k]=e_k$. Then
\begin{enumerate}
\item each $a_k$ is an idempotent of norm at least one,
\item $a_ka_j=0$ for all $k\neq j$,
\item $a_1+a_2+\cdots+a_{n-1}=\one_0$,
\item $\{a_k\}_{k=1}^{n-1}$ forms a basis for $\ell_0^1(\zn\slash n\zn)$,
\item Each non-zero idempotent of $\ell_0^1(\zn\slash n\zn)$ is of the form 
\begin{equation*}
\ve_1a_1+\ve_2a_2+\cdots+\ve_{n-1}a_{n-1}\quad \te{for} \quad (\ve_1,\ve_2,\ldots,\ve_{n-1})\in\mc{B}^{n-1},
\end{equation*}
\item $\ell_0^1(\zn\slash n\zn)$ has precisely $2^{n-1}$ idempotents.
\end{enumerate}
\end{proposition}

\section{$L^p$-operator norms on $\C^n$}

We have shown that the conclusion of Theorem \ref{M(A)isLp} can go either way when dropping the two hypotheses.  
As shown in Subsection \ref{UTM}, for each $p \in [1,\infty)$, the algebra of strictly upper triangular matrices acting on $\ell_2^p$, $T_2^p$, is a degenerate $L^p$-operator algebra without a cai and yet its multiplier algebra is nondegenerately representable on an $L^q$-space for any $q \in [1, \infty)$. In contrast, as shown in Subsection \ref{AugId}, $F^p_{0}(\zn\slash 3\zn)$ is a degenerate $L^p$-operator algebra without a cai, but when $p \in [1,p_0] \cup[p_0', \infty)$, its multiplier algebra cannot be isometrically represented on an $L^q$-space for any $q \in [1, \infty)$. At this time, it is unclear what conditions are necessary to guarantee that the multiplier algebra of a degenerate $L^p$-operator algebra without a cai is again an $L^p$-operator algebra. 

Given that the norms of $L^p$-operator algebras are not unique, a natural question is to characterize norms on $\C^n$ that make it into an $L^p$-operator algebra. This problem is already interesting for $p=1$, for which we know of at least two norms on $\C^n$ which make it an $L^1$-operator algebra:
\begin{enumerate}
\item $\C^n$ with the supremum norm is an $L^1$-operator algebra acting on $\ell_n^1$
via multiplication operators.
\item Let $\mathcal{F} \colon \ell^1(\Z/n\Z) \to C(\Z/n\Z)$ the Fourier transform. 
Then we have algebra isomorphisms
\[
\C^n \cong C(\Z/n\Z) \cong  \mathcal{F}^{-1}(C(\Z/n\Z)) = \ell^1(\Z/n\Z) 
\]
which make $\C^n$ an $L^1$-operator algebra with norm coming 
from the identification with $\ell^1(\Z/n\Z)$. 
\end{enumerate}
\begin{ques}
Are these the only two norms that make $\C^n$ an $L^1$-operator algebra?
\end{ques}
Theorem \ref{NotLpOpAlg} exhibits a norm in $\C^2$ for which $\C^2$ fails to be an $L^q$-operator algebra for any $q \in [1, \infty)$, which is to say the identification of $M_0^1(\Z/3\Z)$ with $\C^2$ carries neither of the norms listed above. This leads to a more general question.
\begin{ques}\label{nextpaperquestion}
Let $n\in \Z_{>1}$ and consider $\C^n$ with pointwise multiplication. Is it possible to find all the norms on $\C^n$ that make it an $L^1$-operator algebra? An $L^p$-operator algebra for some $p\in[1,\infty)$? 
\end{ques}

\subsection*{Acknowledgments}
The authors would like to thank both David Blecher and Eusebio Gardella for their encouragement and discussions particularly regarding Question \ref{nextpaperquestion} above. We are also thankful to Hannes Thiel for pointing out several references and for encouraging us to generalize our main results to $p>1$. The second named author would also like to thank N. Christopher Phillips, who read earlier versions of Section \ref{secMA} and gave suggestions to improve the presentation. Finally, we would like to thank the anonymous referee for a careful and thorough reading of the initial draft and for sending useful references and suggestions that greatly improved some content in this paper. In particular, we thank the referee for pointing out that it made more sense to look at the augmentation ideal of $F^p(G)$ instead of $F^p_{\op{r}}(G)$ as we had initially considered. 
\bibliographystyle{plain}
\bibliography{Lp-mult.bib}

\end{document}